\RequirePackage{fix-cm}
\documentclass[smallextended]{svjour3}       % onecolumn (second format)
\smartqed  % flush right qed marks, e.g. at end of proof
\usepackage{
amsmath,
amssymb,
hyperref,
mathrsfs,
graphicx,
stmaryrd,
amscd,
 amsfonts, stmaryrd,latexsym, xargs,lineno}
\usepackage{color}

\usepackage[colorinlistoftodos]{todonotes}
 \presetkeys{todonotes}%
{inline,backgroundcolor=gray!20,bordercolor=gray!30}{}
\tikzset{/tikz/notestyleraw/.append style={text=black}}

\newtheorem{thm}{Theorem}[section]

\newtheorem{lem}[thm]{Lemma}
\newtheorem{defn}[thm]{Definition}
\newtheorem{prop}[thm]{Proposition}

\newtheorem{rmk}[thm]{Remark}
\newcommand{\be}{\begin{eqnarray}}
\newcommand{\ee}{\end{eqnarray}}
\newcommand{\beq}{\begin{equation}}
\newcommand{\eeq}{\end{equation}}
\newcommand{\ben}{\begin{eqnarray*}}
\newcommand{\een}{\end{eqnarray*}}
\newcommand{\beal}{\begin{aligned}}
\newcommand{\enal}{\end{aligned}}

\newcommand{\eps}{\varepsilon}

\newcommand{\lb}{\lambda}

\newcommand{\T}{\mathbb{T}}
\newcommand{\R}{\mathbb{R}}
\newcommand{\Q}{\mathbb{Q}}

\newcommand{\Z}{\mathbb{Z}}

\newcommand{\om}{\omega}
\newcommand{\Om}{\Omega}

\newcommand{\dt}{\delta}

\newcommand{\cM}{\mathcal{M}}

\newcommand{\cA}{\mathcal{A}}
\newcommand{\cL}{\mathcal{L}}

\newcommand{\cU}{\mathcal{U}}
\newcommand{\cH}{\mathcal{H}}

\newcommand{\cF}{\mathcal{F}}

\newcommand{\wh}{\widehat }
\newcommand{\wt}{\widetilde }
\newcommand{\cI}{\mathcal{I}}

\begin{document}

\title{Parameterized viscosity solutions of convex Hamiltonian systems with time periodic damping}
%\subtitle{}

\titlerunning{viscosity solutions of Hamiltonian systems with time periodic damping}        % if too long for running head

\author{
Ya-Nan Wang  
\and Jun Yan     
\and Jianlu Zhang
}

\authorrunning{Y-N Wang, J Yan, J Zhang} % if too long for running head

\institute{Ya-Nan Wang\at
School of Mathematical Sciences, Nanjing Normal University, Nanjing, 210097, China\\
\email{yananwang@njnu.edu.cn}
\and
Jun Yan \at
School of Mathematical Sciences, Fudan University \& Shanghai Key Laboratory for Contemporary Applied Mathematics, Shanghai 200433, China\\
\email{yanjun@fudan.edu.cn}
\and 
Jianlu Zhang \at
              Hua Loo-Keng Key Laboratory of Mathematics \& Mathematics Institute, Academy of Mathematics and systems science, Chinese Academy of Sciences, Beijing 100190, China \\
              %Tel: +86-182-1038-3625\\
              \email{jellychung1987@gmail.com}           %  \\
}

\date{Received: date / Accepted: date}
%\linenumbers

\maketitle

\begin{abstract}
In this article we develop an analogue of Aubry Mather theory for time periodic dissipative equation
\[
\left\{
\begin{aligned}
\dot x&=\partial_p H(x,p,t),\\
\dot p&=-\partial_x H(x,p,t)-f(t)p
\end{aligned}
\right.
\]
with $(x,p,t)\in T^*M\times\T$ (compact manifold $M$ without boundary). We discuss the asymptotic behaviors of viscosity solutions of associated Hamilton-Jacobi equation
\[
\partial_t u+f(t)u+H(x,\partial_x u,t)=0,\quad(x,t)\in M\times\T
\]
w.r.t. certain parameters, and analyze the meanings in controlling the global dynamics. We also discuss the prospect of applying our conclusions to many physical models.

\end{abstract}

%\tableofcontents

%\vspace{10pt}

\section{Introduction}\label{s1}

 For a smooth compact Riemannian manifold $M$ without boundary, the Hamiltonian $H$ is usually characterized as a $C^{r\geq 2}-$smooth function on the cotangent bundle $T^*M$, with the associated Hamilton equation defined by 
 \be\label{eq:ham}
  {\sf (Conservative)\quad }\left\{
\begin{aligned}
\dot x&=\partial_p H(x,p)\\
\dot p&=-\partial_x H(x,p)
\end{aligned}
\right.
\ee
for each initial point $(x,p)\in T^*M$. From the physical aspect, the Hamiltonian equation describes the movement of particles with conservative energy, since the Hamiltonian $H(x,p)$ verifies to be a {\sf First Integral} of (\ref{eq:ham}). In particular, if
%The flow $\varphi_{H}^t$ of (\ref{eq:ham}) precisely describes the whole motion of the macroscopic object, starting from any initial condition $(x,p)\in T^*M$. Therefore, to find a mathematical expression of $\varphi_H^t$ becomes the central target of dynamic mathematicians in the three hundred years since the age of Newton, unitil the chaos phenomenon was firstly revealed by Poincar\'e, in his research on the generalized solution of  {\sf Three Body Problem} \cite{P}. That ends the endeavor to solve $\varphi_H^t$ analytically and enlighten a new direction to catagorize certain {\sf invariant sets} and verify their topological stability in local regions. This transformation leads to a prosper of quanlitative theory in the last centrury.\\
the potential periodically depends on the time $t$ (for systems with periodic propulsion or procession),  we can introduce an augmented Hamiltonian 
\be
\wt H(x,p,t,I)=I+H(x,p,t),\quad\quad(x,p,t,I)\in T^*M\times T^*\T
\ee
such that the associated Hamiltonian equation
 \be\label{eq:ham-aug}
 {\sf (Conservative)\quad }\left\{
\begin{aligned}
\dot x&=\partial_p H(x,p,t)\\
\dot p&=-\partial_x H(x,p,t)\\
\dot t&=1\\
\dot I &=-\partial_t H(x,p,t)
\end{aligned}
\right.
\ee still preserves $\wt H$. \medskip

However, the realistic motion of the masses inevitably sustains a dissipation of energy, due to the friction from the environment, e.g. the wind, the fluid, interface etc. That urges us to make rational modification of previous equations. 
%{\color{blue}
%The contact systems can be used to describe the general dissipative phenomena. 
%In \cite{Su,WWY1,WWY2,WY}, the authors introduce a variational principle called implicit variational principle for the general contact systems, by which the viscosity solution of the corresponding Hamilton-Jacobi equation can be found.
%In \cite{WWY3}, the Aubry-Mather theory for contact Hamiltonian systems was developed under the moderating increasing condition. In \cite{CCJW}, the authors developed another variational formulation for the contact systems, which is an explicit form with nonholonomic constraints. Moreover, the asymptotic behaviors of the Hamilton-Jacobi equation as the contact structure vanishing was investigated in \cite{CCIZ,WYZ}.
%}
In the current paper, the damping is assumed to be time-periodically proportional to the momentum. Precisely,  we modify \eqref{eq:ham-aug} into
% In Sec. \ref{sp}, we will see wide applications  of such a case.
%, e.g. astronomy \cite{CC}, models of transport \cite{WL}, biological physics \cite{Ca2} and economics \cite{B} (see Sec. \ref{sp} for a selection of typical models).\\
%\be\label{eq:ham-main}
%\wh H(x,p,t,I, u)=\wt H(x,p,t,I)+f(t)u-\alpha
%\ee
%for $(x,p,t,I,u)\in T^*M\times T^*\T\times\R$ and {\sf initial energy} $\alpha\in\R$ with the associated  equation 
 \be\label{eq:dis}
 {\sf (Dissipative)\quad }\left\{
\begin{aligned}
\dot x&=\partial_p H(x,p,t)\\
\dot p&=-\partial_x H(x,p,t)-f(t)p\\
\dot t&=1\\
\dot I &=-\partial_t H(x,p,t)-f'(t)u-f(t)I\\
\dot u&=\langle H_p,p\rangle-H+\alpha-f(t)u
\end{aligned}
\right.
\ee
with $\alpha\in\R$ being a constant of initial energy and $f\in C^{r\geq 2}(\T:=\R\slash[0,1], \R)$.  Notice that the former three equations of (\ref{eq:ode1}) is decoupled with the latter two, so we can denote the flow of the former three equations in \eqref{eq:dis} by $\varphi_H^t$ and by $\wh\varphi_{ H}^t$ the flow of the whole \eqref{eq:dis}. The following individual cases of $f(t)$ will be considered:
\begin{itemize}
%\item $f:\T\rightarrow\R$ is $C^{r}$ smooth, and 
\item {\bf (H0$^-$)}  $[f]:=\int_0^1f(t)dt>0$
%\hspace{50pt}(dissipative) 
\item  {\bf (H0$^+$)} $[f]<0$
%\hspace{103pt}(accelerative)
\item {\bf (H0$^0$)} $[f]=0$ 
%\hspace{110pt}(critical)
\end{itemize}

Besides, we 
%
%We will see, the augmented equations indeed supply us more information about the damping mechanism. \\
%
%Now we 
propose the following {\sf standing assumptions} for the Hamiltonian:
\begin{itemize}
\item {\bf (H1)} {\sf [Smoothness]} $H:TM\times\T\rightarrow\R$ is $C^{r\geq 2}$ smooth;
\item {\bf (H2)} {\sf [Convexity]} For any $(x,t)\in M\times\T$, $H(x,\cdot,t)$ is strictly convex on $ T_x^*M$;
\item {\bf (H3)} {\sf [Superlinearity]} For any $(x,t)\in M\times\T$, $\lim_{|p|_x\rightarrow +\infty}H(x,p,t)/|p|_x=+\infty$ where $|\cdot|_x$ is the norm deduced from the Riemannian metric.
\item {\bf (H4)} {\sf [Completeness]} For any $(x,p,\theta)\in T^*M\times\T$, the flow $\varphi_H^t(x,p,\theta)$ exists for all $t\in\R$.
\end{itemize} 
\begin{rmk}\label{rmk:pro}
\begin{itemize}
\item[i)] As we can see, the three different cases of {\bf (H0)} respectively leads to a {\sf dissipation, acceleration and periodic conservation} of energy  along $\wh\varphi_H^t$ in the forward time, if we take
\be\label{eq:ham-main}
\wh H(x,p,t,I, u)=\wt H(x,p,t,I)+f(t)u-\alpha.
\ee
This is because $\dfrac d{dt}\wh H=-f(t)\wh H$.
\item {\bf (H1-H3)} are usually called {\sf Tonelli conditions}. As for {\bf (H4)}, the completeness of $\varphi_H^t$ is actually equivalent to the completeness of $\wh\varphi_H^t$. A sufficient condition to {\bf (H4)} is the following:
\[
|H_x|\leq \kappa (1+|p|_x) \text{ for all }(x,p,t)\in T^*M\times\T
\]
for some  constant $\kappa$.
%In special occasions, the following condition is also required:\smallskip
%\begin{itemize}
%\item {\bf (H4)} {\sf [Completeness]} For any $(x,p,\theta)\in T^*M\times\T$, the flow $\varphi_H^t(x,p,\theta)$ exists for all $t\in\R$;
%\end{itemize}\smallskip
\item[ii)] Observe that the time-1 map $\varphi_H^1:\{(x,p,t=0)\}\rightarrow\{(x,p,t=0)\}$ is  {\sf conformally symplectic}, i.e.
 \[
 (\varphi_H^1)^*dp\wedge dx=e^{[f]}dp\wedge dx.
 \]
Such maps have wide applications in astronomy \cite{CC}, optimal transport \cite{WL}, biological physics \cite{Ca2} and economics \cite{B} etc (see Sec. \ref{sp} for more details).
%\textr{
%  Specially, the following assumption 
%\begin{itemize}
%\item {\bf (H5)} $|H_x|\leq \kappa (1+|p|_x)$ for all $(x,p,t)\in T^*M\times\T$;
%\end{itemize}
%can be added, to make the flow $\wh\varphi_{H}^t$ of (\ref{eq:dis}) {\sf complete}. This is because for any $t\in\R$, there exists suitable constants $K,\ M>0$ such that 
%\[
%|\dot p|\leq K+M|p|_x
%\]
%which implies $p-$variable of $\wh\varphi_{H}^t$ is complete. Consequently, the $u$ and $I$ variables of $\wh \varphi_{ H}^t$ have to  be complete ae well.}
%\smallskip
\end{itemize}
\end{rmk}
In the following, we will explore the dynamics of \eqref{eq:dis} by using an analogue of Aubry Mather theory \cite{Mat} or weak KAM theory \cite{Fa}. Similar research was exploited in \cite{CCJW,WWY2,WWY3}, for generalized $1^{st}$ order PDEs. The current paper have a lot of similarities in the methodology with these works.

%Notice that the previous three items of (\ref{eq:ode1}) is decoupled with the latter two, and by combining the previous three items we get (\ref{eq:ode0}). Moreover,  the following properties can be perceived:
%\begin{itemize}
%\item {\sf (completeness)} Since $0<\dt\leq f(t)\leq \Dt$ and $\|V\|_{C^1}\leq K$, 
%\[
%|\dot p|\leq K+M|p|
%\]
%which implies $p-$variable of each trajectory is always complete. Similarly we  can prove that $u$ is  complete and then $I$ is complete.\\
%\item {\sf (dissipation)} $\frac d{dt}F=-f(t)F$ along each trajectory,  $\lim_{t\rightarrow+\infty}F=0$ for all initial points.\\
%\end{itemize}
%\begin{rmk}
%These properties help us to locate the invariant sets by means of a variational principle, which will be discussed in Sec. \ref{s3} with details. 
%\end{rmk}

\subsection{Variational Principle and Hamilton Jacobi equation} As the dual of the Hamiltonian, the {\sf Lagrangian} can be defined as the  following 
\be\label{eq:led}
L(x,v,t):=\max_{p\in T^*_xM} \langle p,v\rangle-H(x,p,t),\quad (x,v,t)\in TM\times\T.
\ee
of which the maximum is achieved for $v=H_p(x,p,t)\in T_xM$, once {\bf (H1-H3)} are assumed. Therefore, the {\sf Legendre transformation}
\beq
\cL: T^*M\times\T\rightarrow TM\times\T, \quad\text{via } (x,p,t)\rightarrow (x, H_p(x,p,t),t)
\eeq
%\[
%\text{$\Big($resp. $\wh{\cL}: T^*M\times T^*\T\times\R\rightarrow TM\times T^*\T\times\R$},
%\]
%\[
%\quad\quad\quad \text{via } (x,p,t,I,u)\rightarrow (x, \partial_pH(x,p,t),t,I,u)\Big)
%\]
is a diffeomorphism. Notice that the Lagrangian $L:TM\times\T\rightarrow\R$ is also $C^r-$smooth and convex, superlinear in $p\in T_x M$, so by a slight abusion of notions we say it satisfies {\bf (H1-H3)} as well. 
As a conjugation of $\varphi^t_H$, the {\sf Euler-Lagrangian flow} $\varphi_L^t$ is defined by 
\beq\label{eq:e-l}\tag{E-L}
\left\{
\begin{aligned}
&\dot x=v,\\
&\frac d{dt}L_v(x,v,t)=L_x(x,v,t)-f(t)L_v(x,v,t),\\
&\dot t=1.
\end{aligned}
\right.
\eeq
 It has equivalently effective in exploring the dynamics of  (\ref{eq:dis}), and the completeness of $\varphi_L^t$ is equivalent to the completeness of $\varphi_L^t$. 
To present \eqref{eq:e-l} a variational characterization, we introduce a minimal variation on the absolutely continuous curves with fixed endpoints
 \[
 h_\alpha^{s,t}(x,y)=\inf_{\substack{\gamma\in C^{ac}([s,t],M) \\ \gamma(s)=x,\gamma(t)=y}}\int^t_se^{F(\tau)}(L(\gamma,\dot{\gamma},\tau)+\alpha)\mbox{d}\tau,
 \]
where $F(t)=\int^t_0f(\tau)\mbox{d}\tau$ and $\alpha\in\mathbb{R}$. 
 It is a classical result in the calculus of variations, the infimum is always available for all  $s<t\in\R$, which is actually $C^r-$ smooth and satisfies (\ref{eq:e-l}), once {\bf (H4)} is assumed (due to the {\sf Weierstrass Theorem} in \cite{Mat} or Theorem 3.7.1 of \cite{Fa}).
\begin{thm}[main 1]\label{thm:1} For  $f(t)$ satisfying {\bf  (H0$^-$)}, $H(x,p,t)$ satisfying {\bf (H1-H4)} and any $\alpha\in\R$, the following 
\[
u_\alpha^-(x,t):=\inf_{\substack{\gamma\in C^{ac}((-\infty,t],M)\\\gamma(t)=x}}\int^t_{-\infty}e^{F(s)-F(t)}(L(\gamma(s),\dot{\gamma}(s),s)+\alpha)\mbox{d}s
\]
is well defined for $(x,t)\in M\times\R$ and satisfies 
 \begin{enumerate}
% \item {\sf (Convergency)} $
%u_\alpha^-(x,t):=\lim_{s\rightarrow-\infty}\cT_{s,t,\alpha}^-\psi$ uniquely exists whatever $\psi$ is; Moreover, $u_\alpha^-$ has an explicit expression:
%\[
%u_\alpha^-(x,t):=\inf\bigg\{\int^t_{-\infty}e^{F(s)-F(t)}(L(\gamma(s),\dot{\gamma}(s),s)+\alpha)\mbox{d}s\Big |\gamma\in C^{ac}((-\infty,t],M),\gamma(t)=x\bigg\},
%\]
%for $(x,t)\in M\times\R$.
 \item {\sf (Periodicity)} $u_\alpha^-(x,t+1)=u_\alpha^-(x,t)$ for any $x\in M$ and $t\in\R$. By taking $\bar t\in[0,1)$ with $t\equiv \bar t\  (mod\ 1)$ for any $t\in\R$, we can interpreted $u_\alpha^-$ as a function on $M\times\T$.
 \item {\sf (Lipschitzness)} $u_\alpha^-:M\times\T\rightarrow \R$ is Lipschitz, with the Lipschitz constant depending on $L$ and $f$;
\item {\sf (Domination\footnote{Any function $\om\in C(M\times\T,\R)$ satisfying (\ref{eq:dom}) is called a {\sf (viscosity) subsolution} of (\ref{eq:sta-hj}) and denoted by $\om\prec_f L+\alpha$. 
%Hence, we have $u_\alpha^-\prec_f L+\alpha$;
})} For any absolutely continuous curve $\gamma:[s,t]\rightarrow M$ connecting $(x,\bar s)\in M\times \T$ and $(y,\bar t)\in M\times\T$, we have 
\be\label{eq:dom}
e^{F(t)}u_\alpha^-(y,\bar t)-e^{F(s)}u_\alpha^-(x,\bar s)\leq \int_s^t e^{F(\tau)}\Big(L(\gamma,\dot\gamma,\tau)+\alpha\Big)d\tau.
\ee
%with $\{t\},\ \{s\}\in[0,1)$ are uniquely established by $t\equiv \{t\}\  (mod\ 1)$,  $s\equiv \{s\}\  (mod\ 1)$;
\item {\sf (Calibration)} For any $(x,\theta)\in M\times\T$, there exists a {\sf backward calibrated curve} curve $\gamma_{x,\theta}^-:(-\infty,\theta]\rightarrow M$, $C^r-$smooth and ending with $\gamma_{x,\theta}^-(\theta)=x$, such that for all $s\leq t\leq\theta$, we have 
\be\label{eq:cal}
& &e^{F(t)}u_\alpha^-(\gamma_{x,\theta}^-(t),\bar t)-e^{F( s)}u_\alpha^-(\gamma_{x,\theta}^-(s),\bar s)\nonumber\\
&=&\int_s^t e^{F(\tau)}\Big(L(\gamma_{x,\theta}^-,\dot\gamma_{x,\theta}^-,\tau)+\alpha\Big)d\tau. 
\ee
%\footnote{Any Lipschitz continuous curve $\gamma:\R\rightarrow M$ satisfying (\ref{eq:cal}) for any $s\leq t\in\R$ is called a globally calibrated curve;}.
\item {\sf (Viscosity)} $u_\alpha^-:M\times\T\rightarrow\R$ is a viscosity solution of the following {\sf Stationary Hamilton-Jacobi equation} (with time periodic damping):
\beq\label{eq:sta-hj}\tag{HJ$_+$}
\partial_t u+f(t)u+H(x,\partial_x u,t)=\alpha,\quad(x,t)\in M\times\T,\ \alpha\in\R.
\eeq
%\be\label{eq:sta-hj}
%\partial_tu(x,t)+H(x,\partial_x u,t)+f(t) u=\alpha,\quad (x,t)\in M\times\R.
%\ee
%\textr{
%\item {\sf (Differentiability)} $u_\alpha^-$ is differentiable at $\gamma_{x,\theta}^-(s)$ for $s<\theta$, where $\gamma_{x,\theta}^-:s\in(-\infty,\theta]\rightarrow M$ is a backward calibrated curve ending with $x$; Moreover, we have 
%\[
%(\gamma_{x,\theta}^-(s),\dot\gamma_{x,\theta}^-(s),\bar s)=\cL\Big(\gamma_{x,\theta}^-(s),d_xu_\alpha^-(\gamma_{x,\theta}^-(s),s),\bar s\Big),\quad\forall s<\theta.
%\]
%}
 \end{enumerate}
\end{thm}

\begin{thm}[main 1']\label{cor:1}

For $f(t)$ satisfying {\bf (H0$^0$)} and $H(x,p,t)$ satisfying {\bf (H1-H4)}, there exists a unique $c(H)\in\R$ {\sf (Ma\~n\'e Critical Value)} such that 
\begin{equation}\label{eq:cv}
u^-_{z,\bar{\varsigma}}(x,\bar{t}):=\varliminf_{\substack{\bar{\varsigma}\equiv\varsigma, \bar{t}\equiv t(mod\; 1) \\ t-\varsigma\to+\infty}}\bigg(\inf_{\substack{\gamma\in C^{ac}([\varsigma,t],M) \\\gamma(\varsigma)=z,\gamma(t)=x }}\int^t_{\varsigma}e^{F(\tau)-F(t)}\big(L(\gamma,\dot{\gamma},\tau)+c(H)\big)\mbox{d}\tau\bigg)	
\end{equation}
%\be
%& &h_{c(H)}^\infty\big((z,\bar \varsigma),(y,\bar t)\big)\nonumber\\
%&:=&\liminf_{\substack{t-\bar t\rightarrow+\infty\\
%\varsigma-\bar \varsigma\rightarrow+\infty
%n\rightarrow+\infty\\
%=n\in\Z}\\
%\rightarrow+\infty
%}}\bigg(\inf_{\substack{\gamma\in C^{ac}([\varsigma,t],M)\\ \gamma(\varsigma)=z,\gamma(t)=y }}\int_\varsigma^te^{F(\tau)-F(t)}\Big(L(\gamma,\dot\gamma,\tau)+c(H)\Big)d\tau\bigg)
%\ee
is well defined on $M\times\mathbb{T}$ (for any fixed $(z,\bar{\varsigma})\in 
M\times\T$) and satisfies
%for any $\phi$, $u_\phi^-(x,t):=\varliminf_{\substack{n\in\Z_+\\n\rightarrow+\infty}}\cT_{-n,t}^{c(H),-}\phi(x)$ exists and satisfies:
\begin{enumerate}
% \item {\sf (Periodicity)} $u_\phi^-(x,t+1)=u_\phi^-(x,t)$ for any $x\in M$ and $t\in\R$;
\item {\sf (Lipschitzness)} $u_{z,\bar{\varsigma}}^-:M\times\T\rightarrow \R$ is Lipschitz.
\item {\sf (Domination)} 
%for any fixed $(x,\{s\})\in M\times\T$, the function 
%\[
%u_{x,s}^-:=h_{c(H)}^\infty\big((x,\{s\}),\cdot\big):M\times\T\rightarrow\R
%\]
For any Lipschitz continuous curve $\gamma:[s,t]\rightarrow M$ connecting $(x,\bar s)\in M\times \T$ and $(y,\bar t)\in M\times\T$, we have 
\be\label{eq:dom-c}
e^{F(t)}u_{z,\bar{\varsigma}}^-(y,\bar t)-e^{F(s)}u_{z,\bar{\varsigma}}^-(x,\bar s)
\leq \int_s^t e^{F(\tau)}\Big(L(\gamma,\dot\gamma,\tau)+c(H)\Big)d\tau.
\ee
Namely, $u_{z,\bar{\varsigma}}^-\prec_f L+c(H)$.
\item {\sf (Calibration)} For any $(x,\theta)\in M\times\T$, there exists a $C^r$ curve $\gamma_{x,\theta}^-:(-\infty,\theta]\rightarrow M$ with $\gamma_{x,\theta}^-(\theta)=x$, such that for all $s\leq t\leq\theta$, we have 
\be\label{eq:cal-c}
& &e^{F(t)}u_{z,\bar{\varsigma}}^-(\gamma_{x,\theta}^-(t),\bar t)-e^{F( s)}u_{z,\bar{\varsigma}}^-(\gamma_{x,\theta}^-(s),\bar s)\nonumber\\
&=&\int_s^t e^{F(\tau)}\Big(L(\gamma_{x,\theta}^-,\dot\gamma_{x,\theta}^-,\tau)+c(H)\Big)d\tau. 
\ee
\item {\sf (Viscosity)} $u_{z,\bar{\varsigma}}^-$ is a viscosity solution of 
\beq\label{eq:sta-hj2}\tag{HJ$_0$}
\partial_t u+f(t)u+H(x,\partial_x u,t)=c(H),\quad(x,t)\in M\times\T.
\eeq
%\[
%\partial_tu(x,t)+H(x,\partial_x u,t)+f(t) u=c(H),\quad (x,t)\in M\times\R.
%\]
\end{enumerate}
\end{thm}

Following the terminologies in \cite{Fa,MS}, it's appropriate to call the function given in Theorem \ref{thm:1} (resp. Theorem \ref{cor:1}) a {\sf weak KAM solution}. Such a solution can be used to pick up different types of invariant sets with variational meanings of \eqref{eq:dis}:
\begin{thm}[main 2]\label{thm:2}
For $f(t)$ satisfying {\bf (H0$^-$)}, $H(x,p,t)$ satisfying {\bf (H1-H4)} and any $\alpha\in\R$, we can get the following sets: \smallskip
\begin{itemize}
\item {\sf (Aubry Set)} $\gamma:\R\rightarrow M$ is called {\sf globally calibrated}, if for any $s<t\in\R$, (\ref{eq:cal}) holds on $[s,t]$. There exists a $\varphi_L^t-$invariant set defined by
\[
\wt\cA:=\{(\gamma(t),\dot\gamma(t),\bar t)\in TM\times\T|\gamma \text{ is globally calibrated}\}
\]
%\[
%\Big(\text{resp. }  \wh\cA:=\Big\{\Big(\gamma,\dot\gamma,\{t\},\partial_tu_\alpha^-(\gamma(t),\{t\}) ,u_\alpha^-(\gamma(t),\{t\})\Big)\in TM\times T^*\T\times\T
%\]
%\[
%\Big|\gamma \text{ is globally calibrated}\Big\}\Big)
%\]
with the following properties:
\begin{itemize}
\item $\wt\cA$ %(resp. $\wh\cA$) 
is a Lipschitz graph over the {\sf projected Aubry set} $\cA:=\pi\wt\cA\subset M\times\T$, where $\pi:T^*M\times\T\rightarrow M\times\T$ is the standard projection.
%where $\pi:T^*M\times\T\rightarrow M\times\T$ is the standard projection.
\item $\wt \cA$ %(resp. $\wh\cA$) 
is upper semicontinuous w.r.t. $L:TM\times\T\rightarrow\R$
\item $u_\alpha^-$ is differentiable on $\cA$.\medskip
\end{itemize}
\item {\sf (Mather Set)} Suppose $\mathfrak M_{L}$ is the set of all $\varphi_L^t-$invariant probability measure, then $\tilde{\mu}\in\mathfrak M_L$ is called a {\sf Mather measure} if it minimizes
\[
\min_{\tilde{\nu}\in\mathfrak M_L}\int_{TM\times\T}L+\alpha- f(t)u_\alpha^-\mbox{d}\tilde{\nu}.
\]
Let's denote by $\mathfrak M_m$ the set of all Mather measures. Accordingly, the {\sf Mather set} is defined by
\[
\wt\cM:=\overline{\bigcup\{supp\ \tilde{\mu}|\tilde{\mu}\in\mathfrak M_m\}} 
\]
which satisfies
\begin{enumerate}
\item $\wt\cM\neq\emptyset$ and  $\wt\cM\subset\wt\cA$.
\item $\wt\cM$ is a Lipschitz graph over the {\sf projected Mather set} $\cM:= \pi\wt\cM \subset M \times\T$.
\end{enumerate}
\medskip

%the following properties hold
%\begin{itemize}
%\item $\wt\cM\subset\wt\cA$
%\item $\wt\cM$ weakly converges to $\wt\cM_0$ which is the Mather set of the conservative equation (\ref{eq:ham-aug}) as $f(t)\rightarrow 0_+$ pointwisely.
%\end{itemize}
\item {\sf (Maximal Global Attractor)} Define 
\ben
\wh \Sigma_H^-&:=& \big\{(x,p,\bar s,\alpha-f(s)u-H(x,p,s), u)\in T^*M\times T^*\T\times\R\big|\\
& & \quad u> u_\alpha^-(x,s)\big\}
\een
and
\ben
\wh \Sigma_H^0&:=&\big\{(x,p,\bar s,\alpha-f(s)u-H(x,p,s), u)\in T^*M\times T^*\T\times\R\big|\\
& &\quad u= u_\alpha^-(x,s)\big\},
\een
then $\Om:=\bigcap_{t\geq 0}\wh\varphi_{ H}^t( \wh \Sigma_H^-\cup\wh \Sigma_H^0)$ is the maximal $\wh\varphi_{ H}^t-$invariant set, which satisfies:
%Denote by $\Om$ the union of all the {\sf $\om-$limit sets}\footnote{For any $(x,p,t,I,u)\in T^*M\times\T^*\T\times\R$, the $\om-$limit set of the trajectory starting from it is defined by the set of points $(\bar x,\bar p,\bar t,\bar I,\bar u)\in T^*M\times T^*\T\times\R$ for which there exists a sequence $t_k\in\R$ tending to $+\infty$ as $k\rightarrow+\infty$, such that $\lim_{k\rightarrow+\infty}\varphi_{\wh H}^{t_k}(x,p,t,I,u)=(\bar x,\bar p,\bar t,\bar I,\bar u)$} of equation (\ref{eq:dis}), which is definitely $\varphi_{\wh H}^t-$invariant. Then we get 
\begin{enumerate}
\item  If the $p-$component of $\Om$ is bounded, then the $u-$ and $I-$component of $\Om$ are also bounded. 
\item If $\Om$ is compact, it has to be 
%If there exists a maximal compact $\varphi_{\wh H}^t-$invariant set $\Om\subset T^*M\times T^*\T\times\R$, then it has to be 
a {\sf global attractor} in the sense that for any point $(x,p,\bar{s},I,u)\in T^*M\times T^*\T\times\R$ and any open neighborhood $\cU\supseteq\Om$, there exists a $T_{\Om}(\cU)$ such that for all $t\geq T_{\Om}(\cU)$, $\wh\varphi_{ H}^t(x,p,\bar{s})\in \cU$. 
Besides, the followings hold: 
\begin{itemize}
\item $\Om$ is a maximal attractor set, i.e. it isn't strictly contained in any other global attractor;
%\item $\Om\subset \wh \Sigma_H^-\cup\wh \Sigma_H^0$, where 
%\[
%\wh \Sigma_H^-:= \big\{(x,p,s,-f(s)u-H(x,p,s), u)\big| u> u^-(x,s)\big\}
%\]
%and
%\[
%\wh \Sigma_H^0:=\big\{(x,p,s,-f(s)u-H(x,p,s), u)\big| u= u^-(x,s)\big\}.
%\]
\item $\wh\cA$ is the maximal invariant set contained in $\wh \Sigma_H^0$, where 
\ben
\wh\cA&:=&\Big\{\Big(\cL(x,\partial_x u_\alpha^-(x,s), \bar{s}), \partial_tu_\alpha^-(x,s), u_\alpha^-(x,s)\Big)\in TM\times\\
& &\quad  T^*\T\times\R\Big|(x,\bar{s})\in\cA\Big\}.
\een
%\ben
%\wh\cA&:=&\Big\{\Big(\cL\big(\gamma(t),\dot\gamma(t)\big),\{t\},\partial_tu_\alpha^-(\gamma(t),\{t\}), u_\alpha^-(\gamma(t),\{t\})\Big)\in TM\times\\
%& &\quad  T^*\T\times\R\Big|\gamma \text{ is globally calibrated, }t\in\R\Big\}.
%\een
%\textr{
%\item If $\Om_{\max}:=\bigcap_{t\geq 0}\varphi_{\wh H}^t( \wh \Sigma_H^-\cup\wh \Sigma_H^0)$ is compact, then it's a maximal global attractor (containing all other global attractors)
%}
\end{itemize}
%As a criterion of the compactness of $\Om$, we have
%\begin{itemize}
%\item if the $p-$component of $\Om$ is compact, then the $u-$ and $I-$component of $\Om$ is also compact. 
\end{enumerate}

%\item $\Om$ contains $\wt\cA$ and is contained in $\{(x,p,t)|G(x,p,t)\leq0\}$. generically................
\end{itemize}
\end{thm}

\begin{thm}[main 2']\label{cor:critical}
For $f(t)$ satisfying {\bf (H0$^0$)} and $H(x,p,t)$ satisfying {\bf (H1-H4)}, the Ma\~n\'e Critical Value $c(H)$ has an alternative expression 
\be\label{eq:mea-var}
-c(H)=\dfrac{\inf_{\tilde{\mu}\in\mathfrak M_{ L}}\int_{TM\times\T} e^{F(t)}L(x,v,t)\mbox{d}\tilde{\mu}}{\int_0^1e^{F(t)}\mbox{d}t}.
\ee
Moreover, the minimizer achieving the right side of (\ref{eq:mea-var}) has to be a {\sf Mather measure}. Similarly we can define the {\sf Mather set} $\wt\cM$ as the union of the support sets of all the Mather measures, which is Lipschitz-graphic over the {\sf projected Mather set} $\cM:=\pi\wt\cM$.
\end{thm}
%\begin{rmk}
% In \cite{SZ}, they construct an example to show that $\wt\cM\subsetneq\wt\cA$ could happen. 
%\end{rmk}

\subsection{Parametrized viscosity solutions and asymptotic dynamics} In this section we deal with two kinds of parametrized viscosity solutions with practical meanings. The first case corresponds to a Hamiltonian 
\be\label{eq:ham-par}
\wh H_{\dt}(x,p,t,I, u):=I+ H(x,p,t)+f_\dt(t)u,  
\ee
with $(x,p,\bar{t},I,u)\in T^*M\times T^*\T\times\R$ and $f_\dt\in C^r(\T,\R)$ continuous of $\dt\in\R$. For suitable $\alpha\in\R$,
% belonging to a certain \textr{\sf admissible set} $\cD_{c,\dt}\subset \R$, 
we can seek the weak KAM solution of 
\be\label{eq:hj-par}
\partial_tu_{\dt}(x,t)+H(x,\partial_x u_{\dt},t)+f_\dt(t) u_{\dt}=\alpha
\ee
as we did in previous theorems. Consequently, it's natural to explore the convergence of viscosity solutions w.r.t. the parameter $\dt$:
\begin{thm}[main 3]\label{thm:3}
Suppose $f_\dt$ converges to $f_0$ w.r.t. the uniform norm as $\dt\rightarrow 0_+$ such that $[f_0]=0$ and  the right derivative of $f_\dt$ w.r.t. $\dt$ exists at $0$, i.e.
\be\label{eq:1-jet}
%\int_0^1e^{F_0(t)}
f_1(t):=\lim_{\dt\rightarrow 0_+}\frac{f_\dt(t)-f_0(t)}{\dt}>0.
\ee
If $H(x,p,t)$ satisfying {\bf (H1-H4)}, 
%\begin{enumerate}
%\item if $[f_0]>0$, then there exists $\dt_0=\dt_0([f_0])>0$, such that for any $\alpha\in\R$, the sequence of weak KAM solutions $\{u_\dt^-\}_{\dt\in(0,\dt_0]}$ of (\ref{eq:hj-par}) associated with $f_\dt$ and $\alpha_{0,\dt}\equiv\alpha$ for all $\dt\in(0,\dt_0]$ converges to the weak KAM solution $u_0^-$ of 
%\[
%\partial_tu(x,t)+H(x,\partial_x u,t)+f_0(t) u=\alpha,\quad (x,t)\in M\times\T.
%\]
%\item 
%\begin{itemize}
% \item $[f_0]=0$,
% \item    the right derivative of $f_\dt$ w.r.t. $\dt$ exists at $0$, i.e.
%\be\label{eq:1-jet}
%%\int_0^1e^{F_0(t)}
%f_1(t):=\lim_{\dt\rightarrow 0_+}\frac{f_\dt(t)-f_0(t)}{\dt}>0,
%\ee
%\item the flow $\varphi_H^t$ associated with $H(x,p,t)$ and $f_0(t)$ satisfies {\bf (H4)}, 
%\end{itemize}
%\be
%\frac{d}{d\dt}\bigg|_{\dt=0}f_\dt(t)\geq0,\quad t\in\T,
%\ee
then there exists a unique $c(H)\in\R$ given by (\ref{eq:mea-var}) and a $\dt_0>0$, such that the weak KAM solution $u^-_\dt(x,t)$ of  (\ref{eq:hj-par}) associated with $f_\dt$ and $\alpha_{\dt}\equiv c(H)$ for all $\dt\in(0,\dt_0]$ converges to a uniquely identified viscosity solution of 
%\be\label{eq:hj-par-2}
\be\label{eq:hj-criti}
\partial_tu(x,t)+H(x,\partial_x u,t)+f_0(t) u=c(H),
\ee
which equals 
% is actually the largest viscosity subsolutions of (\ref{eq:hj-par}) (with $\alpha=c(H)$ and $\dt=0$) such that 
\[
\sup\Big\{u\prec_{f_0}L+c(H)\Big|\int_{TM\times\T}e^{F_0(t)} f_1(t)\cdot u(x,t)d
%\footnote{Here $\pi_\sharp\mu$ is the projection of $\mu$, see (\ref{eq:proj-mea} for a precise definition}
\tilde{\mu}\leq 0,\ \forall\; \tilde{\mu}\in\mathfrak M_m(\dt=0)\Big\}
\]
with $F_0(t)=\int_0^tf_0(\tau)\mbox{d}\tau$ and $\mathfrak M_m(0)$ being the set of Mather measures for the system with $\dt=0$.
%\end{enumerate}
%Here 
%\be\label{eq:mea-var}
%c(H)=-\dfrac{\inf_{\mu\in\mathfrak M_{\wt L}}\int_{TM\times\T} \wt Ld\mu}{\int_0^1e^{F_0(t)}dt}
%\ee
%with the time periodic Lagrangian defined by 
% \[
%\wt L(x,v,t):=e^{F_0(t)}L(x,v,t).
% \]
%$\mathfrak M(0)$ is the set of all minimizers of (\ref{eq:mea-var}), i.e. the Mather measure set.
\end{thm}

\begin{rmk}
 The convergence of the viscosity solutions $1^{st}$ order PDEs was earlier discussed in \cite{CCIZ,DFIZ,WYZ,Z}, where the {\sf Comparison principle} was used to guarantee the uniqueness of viscosity solution for \eqref{eq:hj-par}. However, in our case we didn't assume  $f_\dt$ to be nonnegative, which invalidates this principle and brings new new difficulties to prove the equi-boundedness and equi-Lipschitzness of $\{u_\dt^-\}_{\dt>0}$. Fortunately, by analyzing the properties of the {\sf Lax-Oleinik semigroups} we can still overcome these difficulties,
 see Sec. \ref{s4} for more details.

%In Sec. \ref{s4} we will see that for the case $[f_0]=0$, $\wt u^-_0(x,t):=e^{F_0(t)}u_0^-(x,t)$ can be interpreted as a classical weak KAM solution w.r.t. $\wt L(x,v,t)$. In other words, 
%\[
%\partial_t \wt u^-_0(x,t)+e^{F_0(t)}H(x,e^{-F_0(t)}\partial_x\wt u^-_0(x,t),t)=e^{F_0(t)}c(H),\quad a.e. (x,t)\in M\times\T.
%\]
%Unlike the case $[f_\dt]>0$, such a weak KAM solution $\wt u^-_0(x,t)$ is usually not unique.\\
\end{rmk}

%The second conclusion we want to present is about the invariant sets with variational meanings, by means of previously given $u^-(x,t)$.

%\subsection{Parametrized 1-D mechanical models}
The second parametrized problem we concern takes $M=\T$ and a mechanical $H(x,p,t)$. We can involve a cohomology parameter $c\in H^1(\T,\R)$ to 
 \be
\wh H(x,p,t,I,u)=I+\underbrace{\frac1 2(p+c)^2+V(x,t)}_{H(x,p,t)}+f(t)u
\ee
of which $H(x,p,t)$ surely satisfies {\bf (H1-H4)},
then (\ref{eq:dis}) becomes
\be\label{eq:ode1}
 {\sf (Dissipative)\quad } \left\{
\begin{aligned}
\dot x&=p+c\\
\dot p&=-V_x-f(t)p\\
\dot t&=1\\
\dot I &=-V_t-f'(t)u-f(t)I\\
\dot u&=\frac12(p^2-c^2)-V(x,t)-f(t)u.
\end{aligned}
\right.
\ee
In physical models, the former three equations of (\ref{eq:ode1}) is usually condensed into a single equation
\be\label{eq:ode0}
\ddot x+V_x(x,t)+f(t)(\dot x-c)=0,\quad (x,t)\in M\times\T.
\ee
\begin{thm}[main 4]\label{thm:4}
For $f(t)$ satisfying {\bf (H0$^-$)}, the following conclusions hold for equation (\ref{eq:ode0}):
\begin{itemize}
\item For any $c\in H^1(\T,\R)$, there exists a unified {\sf rotation number} of $\wt\cA(c)$, which is defined by 
\[
\rho(c):=\lim_{T\rightarrow+\infty}\frac1 T\int_0^Td\gamma,\quad\forall\ \text{globally calibrated curve } \gamma.
\]
\item $\rho(c)$ is continuous of $c\in H^1(\T,\R)$; Moreover, we have 
\be\label{eq:rot-num-app}
|\rho(c)-c|\leq \varsigma([f])\cdot\|V(x,t)\|_{C^1}
\ee
for some constant $\varsigma$ depending only on $[f]$. Consequently, for any  $p/q\in \Q$  irreducible, there always exists a $c_{p/q}$ such that $\rho(c_{p/q})=p/q$. 
%\item Suppose $p/q\in \Q$ is irreducible and fixed. For open dense $V(x,t)\in C^2(\T^2,\R)$, there exists an interval $[c_-, c_+]$ such that for any $c$ in it, $\wt\cA(c)$ possesses the same rotation number  $ p/q$.
\item There exists an compact maximal global attractor $\Om\subset T^*\T\times T^*\T\times\R$ of the flow $\wh \varphi_{ H}^t$. 
\end{itemize}
\end{thm}
\vspace{10pt}

\noindent{\bf Organization of the article:} The paper is organized as follows: In Sec. \ref{sp}, we exhibit a list of physical models with time-periodic damping. For these models, we state some notable dynamic phenomena and show how these phenomena can be linked to our main conclusions. In Sec. \ref{s2}, we prove Theorem \ref{thm:1} and Theorem \ref{cor:1}. In Sec. \ref{s3}, we get an analogue Aubry Mather theory for systems satisfying {\bf (H0$^-$)} condition, and prove Theorem \ref{thm:2}. Besides, we also prove Theorem \ref{cor:critical} for systems satisfying {\bf (H0$^0$)} condition. In Sec. \ref{s4}, we discuss the parametrized viscosity solutions of (\ref{eq:hj-par}), and prove the convergence of them. In Sec. \ref{s5},  for  1-D mechanical systems with time periodic damping, we prove Theorem \ref{thm:4}, which is related to the dynamic phenomena of the models in Sec \ref{sp}. For the consistency of the proof, parts of preliminary conclusions are postponed to the Appendix.
\vspace{10pt}

\noindent{\bf Acknowledgements:} The first author is supported by Natural Scientific Foundation of China (Grant No.11501437). The second author is supported by National Natural Science Foundation of China (Grant No. 11631006, 11790272) and Shanghai Science and Technology Commission (Grant No. 17XD1400500). The third author is supported by the Natural Scientific Foundation of China (Grant No. 11901560). All the authors are grateful to Prof. Wei Cheng for helpful discussion about the details.

\section{Zoo of practical models}\label{sp}

%\vspace{20pt}

In this section we display a bunch of physical models with time-periodic damping, and introduce some practical problems (related with our main conclusions) around them. 
\subsection{Conformally symplectic systems} For $f(t)\equiv \lb>0$ being constant, we get a so called {\sf conformally symplectic system (or discount system)}. The associated ODE  becomes
\be
 \left\{
\begin{aligned}
\dot x&=\partial_p H(x,p,t),\\
\dot p&=-\partial_x H(x,p,t)-\lb p.
\end{aligned}
\right.
\ee
This kind of systems has been considered in \cite{CCD,DFIZ,MS}, although earlier results on Aubry-Mather sets have been discussed by Le Calvez \cite{LC} and Casdagli \cite{Ca} for $M=\T$. Besides, we need to specify that the Duffing equation with {\sf viscous damping} also conforms to this case, which concerns all kinds of oscillations widely found in electromagnetics \cite{M} and elastomechanics \cite{MH}.

A significant property this kind of systems possess is that 
\[ 
(\varphi_H^1)^*dp\wedge d x= e^{\lb }dp\wedge d x.
\]
When $H(x,p,t)$ is mechanical, the equation usually describes the low velocity oscillation of a solid in a fluid medium (see Fig. \ref{fig1}), which can be formally expressed as
\be
\ddot{x}+\lb\dot x+\partial_x V(x,t)=0, \quad x\in\T,\;\lb>0.
\ee
Chaos and bifurcations topics of this setting has ever been rather popular in 1970s \cite{H}.

\begin{figure}
\begin{center}
\includegraphics[width=9cm]{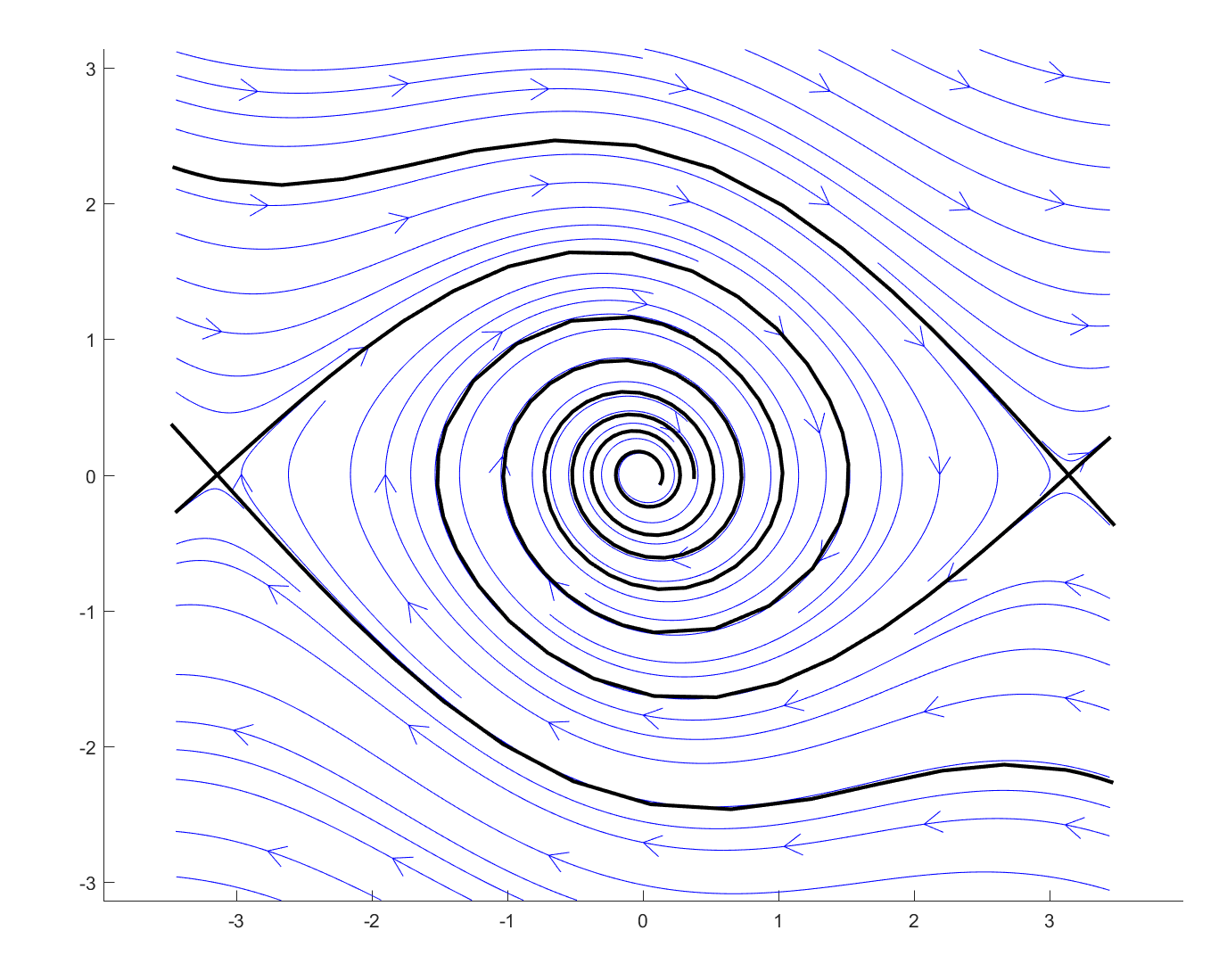}
\caption{A dissipative pendulum with $\lb=1/5$ and $V(t,x)=1-\cos x$.}
\label{fig1}
\end{center}
\end{figure}

\subsection{Tidal torque model} The {\sf tidal torque model} was firstly introduced by \cite{P}, describing the motion of a rigid satellite $S$ under the gravitational influence of a point-mass planet $P$. Due to the internal non-rigidity of the body, a tidal torque will causes a time-periodic dissipative to the motion of $S$, which can be formalized by 
\be\label{eq:tidal}
\ddot x +\eps V_x(x,e,t)+\kappa\eta(e,t)(\dot x-c(e))=0,\quad (x,t)\in\T^2, 
\ee
with the parameter $e$ is the {\sf eccentricity} of the elliptic motion $S$ around $P$. Due to the astronomical observation, $\eps$ is the {\sf equatorial ellipticity} of the satellite and 
\[
\kappa\propto \frac 1 {a^3}\cdot\frac{m_P}{m_S}, 
\]
with $a$ being the {\sf semi-major} and $m_P$ (resp. $m_S$) being the mass respectively.\\

\begin{figure}
\begin{center}
\includegraphics[width=12cm]{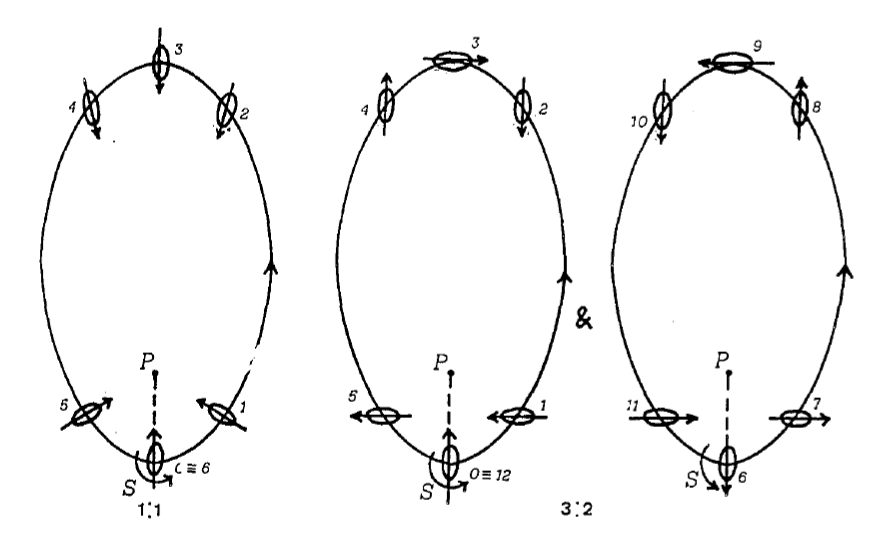}
\caption{A tidal torque model for Moon-Earth and Mercury-Sun.}
\label{fig2}
\end{center}
\end{figure}

Although this model might seem very special, there are several examples in the solar system for which such a model yields a good description of the motion, at least in a first approximation, and anyhow represents a first step toward the understanding of the problem. For instance, in the pairs 
Moon-Earth, Enceladus-Saturn, Dione-Saturn, Rhea-Saturn even Mercury-Sun this model is available. Besides, we need to specify that usually $\kappa\ll \eps$ in all these occasions.\medskip

A few interesting phenomena has been explained by numerical approaches, e.g. the $1:1$ resonance for Moon-Earth system which make the people can only see one side of the Moon from the Earth. However, the Mercury-Sun model shows a different $3:2$ resonance because of the large eccentricity, see Fig. \ref{fig2}.\medskip

Due to Theorem \ref{thm:2} and Theorem \ref{thm:3}, such a resonance seems to be explained by the following aspect: {\bf any trajectory within the global attractor $\Om$ of (\ref{eq:tidal}) has a longtime stability of velocity, namely, the average velocity is close to certain rotation number, or even asymptotic to it.} In Sec. \ref{s5} we will show that variational minimal trajectories indeed match this description.

\begin{rmk}
As a further simplification, a {\sf spin-orbit model} with $\eta(e)$ being a constant is also widely concerned, which is actually a conformally symplectic system. In \cite{CCD} they further discussed the existence of KAM torus for this model and proved the local attraction of the KAM torus.
\end{rmk}

\subsection{Pumping of the swing} The pumping of a swing is usually modeled as a rigid object forced to rotate back and forth at the lower ends of supporting ropes. After a series of approximations and reasonable simplifications, the pumping of the swing can be characterized as a harmonic oscillator with driving and parametric terms \cite{Ca2}. Therefore, this model has a typical meaning in understanding the dynamics of motors.\medskip

As shown in Fig. \ref{fig3}, the length of the ropes supporting the swinger is $l$, and $s$ is the distance between the center of mass of the swinger to the lower ends of the rope. The angle of the supporting rope to the vertical position is denoted by $\phi$, and the angle between the symmetric axis of the swinger and the rope is $\theta$, which varies as $\theta=\theta_0\cos\om t$. So we get the equation of the motion by 
\be\label{eq:swing}
(l^2-2ls\cos\theta+s^2+R^2)\ddot\phi&=&-gl\sin\phi+gs\sin(\phi+\theta)-ls\sin\theta\dot\theta^2\\
& &+(ls\cos\theta-s^2-R^2)\ddot\theta-2ls\sin\theta\dot\theta\dot\phi,\quad \phi\in\T\nonumber
\ee
where $g$ is the gravity index and $mR^2$ is the moment of inertia of the center ($m$ is the mass of swinger). {\bf We can see that by reasonable adjustment of $l,s,\om$ parameters, this system can be  dissipative, accelerative or critical.}\medskip

Notice that numerical research of this equation for $|\phi|\ll1$ has been done by numerical experts in a bunch of papers, see \cite{PGDB} for a survey of that. Those results successfully simulate the swinging at small to modest amplitudes. As the amplitude grows these results become less and less accurate, and that's why we resort to a theoretical analysis in this paper.

\begin{figure}
\begin{center}
\includegraphics[width=2in]{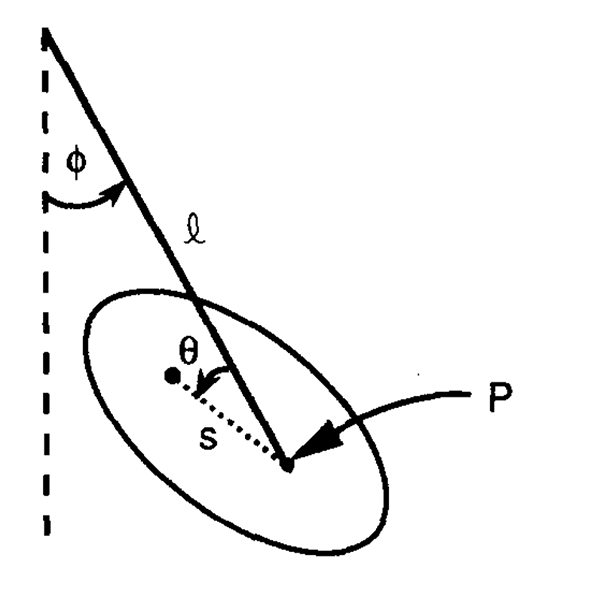}
\caption{A simulation of the pumping of the swing}
\label{fig3}
\end{center}
\end{figure}

\section{Weak KAM solution of (\ref{eq:sta-hj})}\label{s2}

%By a similar analysis as in \cite{Fa}, for any $\phi(x)\in C(M,\R)$, $\cT_{s,t}^{\alpha,-}\phi(x)$
%(resp. $\mathcal{T}^{c(H),-}_{s,t}\phi(x)$) is naturally a viscosity solution of the Cauchy problem
%\[
% [f]>0:\left\{
% \begin{aligned}
% &\partial_tu(x,t)+H(x,\partial_x u,[t])+f(t) u=\alpha,\\
% &u(\cdot,s)=\phi(\cdot),\quad (x,t)\in M\times [s,+\infty),
% \end{aligned}
% \right.
%\]
%\[
%\bigg(resp.\quad 
%[f]=0:\left\{
% \begin{aligned}
% &\partial_tu(x,t)+H(x,\partial_x u,[t])+f(t) u=c(H),\\
% &u(\cdot,s)=\phi(\cdot),\quad (x,t)\in M\times [s,+\infty),
% \end{aligned}
% \right.\bigg)
%\]
%see Appendix \ref{a2} for the proof and other relevant conclusions. To verify the convergence of  $\cT_{s,t}^{\alpha,-}\phi(x)$
%(resp. $\mathcal{T}^{c(H),-}_{s,t}\phi(x)$), 

Due to the superlinearity of $L(x,v,t)$, 
for each $k\geq 0$, there exists $C(k)\geq 0$, such that 
$$
L(x,v,t)\geq k|v|-C(k), k>0,x\in M.
$$
Moreover, the compactness of $M$ implies that for each $k>0$, there exists $C_k>0$ such that 
$$
\max_{\substack{(x,t)\in M\times\mathbb{T}\\|v|\leq k}}L(x,v,t)\leq C_k.
$$
%\subsection{Convergence of $\mathcal{T}^{\alpha-}_{s,t}\phi(x)$ in the condition (H0$^-$) }

\subsection{Weak KAM solution of (\ref{eq:sta-hj}) in the condition \textbf{(H0$^-$)}}

Note that $[f]>0$.  The following conclusion can be easily checked.
% we shall prove the convergence of Lax-Oleinik semigroup $\mathcal{T}^{\alpha-}_{s,t}\phi(x)$ when $f$ satisfies 
%$$
%[f]:=\int^1_0f(\tau)\mbox{d}\tau>0.
%$$
%We give a lemma which will be used frequently.
\begin{lem}\label{Sec3:inequivality}
Assume $t>s$, then
\begin{enumerate}
\item $F(s)-F(t)\leq 2k_0-(t-s-1)[f];$
	\item
	$
	\int^t_se^{F(\tau)-F(t)}\mbox{d}\tau\leq \frac{e^{2k_0+[f]}}{[f]}\big(1-e^{-(t-s)[f]}\big);
	$
	\item $
\int^t_{-\infty}e^{F(\tau)-F(t)}\mbox{d}\tau\leq \frac{e^{2k_0+[f]}}{[f]},
$
\end{enumerate}
where $k_0
=\max_{s\in[0,2]}\big|\int^s_0f(\tau)\mbox{d}\tau\big|
$.
\end{lem}

Now we define a function $u_{\alpha}^-:M\times\mathbb{R}\to\mathbb{R}$ by
\be\label{Sec3:solution}
	u_\alpha^-(x,t)&:=&\inf\int^t_{-\infty}e^{F(s)-F(t)}(L(\gamma(s),\dot{\gamma}(s),s)+\alpha)\mbox{d}s
\ee
where the infimum is taken for all $\gamma\in C^{ac}((-\infty,t],M)$\footnote{aboslutely continuous curves} with $\gamma(t)=x$.
We can easily prove this function is bounded, since 
%\textg{
%since the Tonelli Lagrangian $L: TM\times\T\rightarrow \R$ is bounded below by a constant $-C(0)$, so by
%Lemma \ref{Sec3:inequivality} we derive 
%\ben
%\int^t_{-\infty}e^{F(s)-F(t)}(L(\gamma(s),\dot{\gamma}(s),s)+\alpha)\mbox{d}s&\geq&
%-|C(0)-\alpha|\int^t_{-\infty}e^{F(s)-F(t)}\mbox{d}s\\
%&\geq&-|C(0)-\alpha|\frac{e^{2k_0+[f]}}{[f]},
%\een
%for each Lipschitz continuous curve $\gamma:(-\infty,t]\to M$ with $\gamma(t)=x$.}
$$
-|C(k=0)-\alpha|\cdot\frac{e^{2k_0+[f]}}{[f]}\leq u_{\alpha}^-(x,t)\leq  |C_{k=0}+\alpha|\frac{e^{2k_0+[f]}}{[f]},
$$
where $C(0)$ and $C_0$ have been defined in the beginning of Sec. \ref{s2}. %Later we will see, $u^-_{\alpha}(x,t)$ is actually Lipschitz continuous.

\begin{lem}{\bf [(1) of Theorem \ref{thm:1}]}\label{lem:per}
	$u_{\alpha}^-(x,t)$ is 1-periodic with respect to $t$, i.e.,
	$$
	u_{\alpha}^-(x,t+1)=u_{\alpha}^-(x,t).
	$$
\end{lem}

\proof
By the definition of $u_{\alpha}^-$,
\ben
& &u_{\alpha}^-(x,t+1)\\
&=&\inf_{\gamma(t+1)=x}\bigg\{\int^{t+1}_{-\infty}e^{F(s)-F(t+1)}\big(L(\gamma(s),\dot{\gamma}(s),s)+\alpha\big)\mbox{d}s\bigg\}\\
&= &\inf_{\gamma(t+1)=x}\bigg\{\int^{t}_{-\infty}e^{F(s+1)-F(t+1)}\big(L(\gamma(s+1),\dot{\gamma}(s+1),s+1)+\alpha\big)\mbox{d}s\bigg\}\\
&=&\inf_{\eta(t)=x}\bigg\{\int^{t}_{-\infty}e^{F(s)-F(t)}\big(L(\eta(s),\dot{\eta}(s),s)+\alpha\big)\mbox{d}s\bigg\}\\
&=&u_{\alpha}^-(x,t)
\een
as desired.\qed
\endproof
%\textr{
%Due to Lemma \ref{lem:per}, now we can interpret $u_\alpha^-(x,t)$ both as a function defined on $M\times\T$ and on $M\times\R$, by a slight abuse of notations. }
%{\color{green}
%Then $\mathcal{T}^{\alpha-}_{s,t}\phi(x)$ converges uniformly as $s\rightarrow -\infty$ to $u_\alpha^-(x,t)$ (w.r.t. $(x,t)\in M\times\T$), so $u_\alpha^-:M\times\T\rightarrow\R$ is continuous.
%}

\begin{lem}{\bf [(3) of Theorem \ref{thm:1}]}\label{Sec3:dominant}
Let $\gamma:[s_1,s_2]\to M$ be an absolutely continuous curve. Then,
\be\label{eq:dominant}
	& &e^{F(s_2)}u_{\alpha}^-(\gamma(s_2),s_2)-e^{F(s_1)}u_{\alpha}^-(\gamma(s_1),s_1)\\
	&\leq&
\int^{s_2}_{s_1}e^{F(\tau)}(L(\gamma(\tau),\dot{\gamma}(\tau),\tau)+\alpha)\mbox{d}\tau.\nonumber
\ee

\end{lem}
\proof
Let $\{\gamma_n\}$ be a sequence of absolutely continuous curve from $(-\infty,s_1]$ to $M$ with $\gamma_n(s_1)=\gamma(s_1)$, such that
\[
e^{F(s_1)}u_{\alpha}^-(\gamma(s_1),s_1)=\lim_{n\to \infty}\int_{-\infty}^{s_1}e^{F(\tau)}(L(\gamma_n(\tau),\dot{\gamma}_n(\tau),\tau)+\alpha)\mbox{d}\tau.
\]
Let $\hat{\gamma}_n=\gamma_n*\gamma$ for each $n\in\mathbb{N}$. Hence,
\ben
	e^{F(s_2)}u_{\alpha}^-(\gamma(s_2),s_2)&\leq &\int^{s_2}_{-\infty}e^{F(\tau)}(L(\hat{\gamma}_n(\tau),\dot{\hat{\gamma}}_n(\tau),\tau)+\alpha)\mbox{d}\tau\\
	&\leq& \int^{s_2}_{s_1}e^{F(\tau)}(L(\gamma(\tau),\dot{\gamma}(\tau),\tau)+\alpha)\mbox{d}\tau\\
	& &+\int^{s_1}_{-\infty}e^{F(\tau)}(L(\gamma_n(\tau),\dot{\gamma}_n(\tau),\tau)+\alpha)\mbox{d}\tau.
\een
Taking the limit $n\to\infty$, we derive (\ref{eq:dominant}) is true.\qed
\endproof

\begin{lem}\label{Sec3:minimizer}
For each $(x,t)\in M\times\mathbb{R}$ and $s<t$, it holds
\be\label{eq:dominant_min}
& &e^{F(t)}u_{\alpha}^-(x,t)\\
&=&\inf_{\substack{\gamma\in C^{ac}([s,t],M)\\ \gamma(t)=x }}\bigg\{ e^{F(s)}u_{\alpha}^-(\gamma(s),s)+\int^t_se^{F(\tau)}(L(\gamma(\tau),\dot{\gamma}(\tau),\tau)+\alpha)\mbox{d}\tau \bigg\}.\nonumber
\ee
Moreover, the infimum in (\ref{eq:dominant_min}) can be achieved by a $C^r$ smooth minimizer.
\end{lem}

\proof 
Due to Lemma \ref{Sec3:dominant}, 
$$
e^{F(t)}u_{\alpha}^-(x,t)\leq\inf_{\substack{\gamma\in C^{ac}([s,t],M)\\ \gamma(t)=x
}} \big\{e^{F(s)}u_{\alpha}^-(\gamma(s),s)+\int^t_se^{F(\tau)}(L(\gamma,\dot{\gamma},\tau)+\alpha)\mbox{d}\tau\big\}.
$$
For each $\epsilon>0$, there exists an absolutely continuous curve $\gamma:(-\infty,t]\to M$ with $\gamma(t)=x$, such that
\ben
e^{F(t)}u_{\alpha}^-(x,t)+\epsilon&\geq& \int^t_{-\infty}e^{F(\tau)}(L(\gamma(\tau),\dot{\gamma}(\tau),\tau)+\alpha)\mbox{d}\tau	\\
&=&\int^t_se^{F(\tau)}(L(\gamma(\tau),\dot{\gamma}(\tau),\tau)+\alpha)\mbox{d}\tau\\
& &+\int^s_{-\infty}e^{F(\tau)}(L(\gamma(\tau),\dot{\gamma}(\tau),\tau)+\alpha)\mbox{d}\tau\\
&\geq& e^{F(s)}u_{\alpha}^-(\gamma(s),s)+\int^t_se^{F(\tau)}(L(\gamma(\tau),\dot{\gamma}(\tau),\tau)+\alpha)\mbox{d}\tau.
\een
Hence, (\ref{eq:dominant_min}) proves to be an equality. Therefore, we can find a sequence of absolutely continuous curve $\{\gamma_n\}$ with  $\gamma_n(t)=x$ such that  
$$
e^{F(t)}u_{\alpha}^-(x,t)=\lim_{n\to\infty}\bigg\{e^{F(s)}u_{\alpha}^-(\gamma_n(s),s)+\int^t_se^{F(\tau)}(L(\gamma_n,\dot{\gamma}_n,\tau)+\alpha)\mbox{d}\tau\bigg\}.
$$
Hence, there exists a constant $c$ independent of $n$, such that 
\begin{equation}\label{eq_lem315_bounded}
	\int^t_se^{F(\tau)}(L(\gamma_n(\tau),\dot{\gamma}_n(\tau),\tau)+\alpha)\mbox{d}\tau\leq c.
\end{equation}
Due to Dunford-Petti Theorem (Theorem 6.4 in \cite{DFIZ}), there exists a subsequence $\{\gamma_{n_k}\}$ converging to a curve $\gamma_*$ 
such that
\begin{equation}\label{eq_lem315_dominiant}
	\int^t_se^{F(\tau)}(L(\gamma_*,\dot{\gamma}_*,\tau)+\alpha)\mbox{d}\tau\leq\varliminf_{k\to\infty}\int^t_se^{F(\tau)}(L(\gamma_{n_k},\dot{\gamma}_{n_k},\tau)+\alpha)\mbox{d}\tau.
\end{equation}
Hence, (\ref{eq:dominant_min}) can be achieved at $\gamma_*: [s,t]\to M$, which definitely solves the Euler-Lagrange equation (\ref{eq:e-l}). Due to the { Weierstrass Theorem} in \cite{Mat}, $\gamma_*$ is $C^r$ smooth.\qed
%To show the infimum in (\ref{eq:dominant_min}) can be achieved, let $\{\gamma_n\}$ be a sequence of Lipschitz continuous curve with $\gamma_n(t)=x$ such that  
%$$
%e^{F(t)}u_{\alpha}^-(x,t)=\lim_{n\to\infty}\bigg\{e^{F(s)}u_{\alpha}^-(\gamma_n(s),s)+\int^t_se^{F(\tau)}(L(\gamma_n(\tau),\dot{\gamma}_n(\tau),\tau)+\alpha)\mbox{d}\tau\bigg\}.
%$$
%Hence, there exists a constant $c$ independent of $n$, such that 
%\begin{equation}\label{eq_lem315_bounded}
%	\int^t_se^{F(\tau)}(L(\gamma_n(\tau),\dot{\gamma}_n(\tau),\tau)+\alpha)\mbox{d}\tau\leq c.
%\end{equation}
  \endproof

%\begin{lem}[Dunford-Petti Theorem \cite{DFIZ}]\label{Sec3:compact}
%	Assume $\alpha\in\mathbb{R}, a<b$. Then,
%	$$
%	E_c=
%	\Big\{\gamma\in C^{ac}([a,b],M)\Big|\int^b_ae^{F(s)}(L(\gamma(s),\dot{\gamma}(s),s)%+\alpha)\mbox{d}s<c\Big\}
	%$$
	% is compact in the $C^0$-topology.
	% Moreover, if $\gamma_n\to\gamma$ as $n\to\infty$, then
	% $$
	% \int^b_ae^{F(s)}(L(\gamma(s),\dot{\gamma}(s),s)+\alpha)\mbox{d}s\leq\varliminf_{n\to\infty}\int^b_ae^{F(s)}(L(\gamma_n(s),\dot{\gamma}_n(s),s)+\alpha)\mbox{d}s.
%	 $$
%		\end{lem}
\begin{lem}{\bf [(4) of Theorem \ref{thm:1}]}\label{Sec3:calibrated}
For each $\alpha\in\mathbb{R}$ and $(x,t)\in M\times\mathbb{R}$, there exists a curve $\gamma_{x,t}^-:(-\infty,t]\to M$ with $\gamma_{x,t}^-(t)=x$ such that for each $t_1<t_2\leq t$,
\be\label{eq:calibrated_0}
	& &e^{F(t_2)}u_{\alpha}^-(\gamma_{x,t}^-(t_2),t_2)-e^{F(t_1)}u_{\alpha}^-(\gamma_{x,t}^-(t_1),t_1)\\
	&=&\int^{t_2}_{t_1}e^{F(\tau)}(L(\gamma_{x,t}^-(\tau),\dot{\gamma}_{x,t}^-(\tau),\tau)+\alpha)\mbox{d}\tau.\nonumber
\ee
	%We call $\gamma_*:(-\infty,t]\to M$ a calibrated curve of $u_{\alpha}^-(x,t)$.
\end{lem}
\proof
By Lemma \ref{Sec3:minimizer}, for each $n\in\mathbb{N}$, there exists a sequence of $C^r$ curve $\gamma_n:[t-n,t]\to M$ with $\gamma_n(t)=x$ such that
$$
e^{F(t)}u_{\alpha}^-(x,t)=e^{F(t-n)}u_{\alpha}^-(\gamma_n(t-n),t-n)+\int^t_{t-n}e^{F(\tau)}(L(\gamma_n(\tau),\dot{\gamma}_n(\tau),\tau)+\alpha)\mbox{d}\tau.
$$
It is easy to see for each interval $[a,b]\subset [t-n,t]$ 
\be\label{eq:ca}
	& &e^{F(b)}u_{\alpha}^-(\gamma_n(b),b)-e^{F(a)}u_{\alpha}^-(\gamma_n(a),a)\\
	&=&\int^b_ae^{F(\tau)}(L(\gamma_n(\tau),\dot{\gamma}_n(\tau),\tau)+\alpha)\mbox{d}\tau.\nonumber
\ee
Due to diagonal approach, we derive there exists a subsequence of $\{\gamma_n\}$, denoted by $\{\gamma_{n_k}\}$ and a  curve $\gamma_{x,t}^-:(-\infty,t]\to M$ such that 
$\gamma_{n_k}$ converges uniformly to $\gamma_{x,t}^-$ on each finite subinterval of $(-\infty,t]$.
Taking $k\to\infty$ in (\ref{eq:ca}), due to Dunford Petti Theorem,
\begin{align*}
	&e^{F(b)}u_{\alpha}^-(\gamma_{x,t}^-(b),b)-e^{F(a)}u_{\alpha}^-(\gamma_{x,t}^-(a),a)\\ 
	&=\varliminf_{k\to\infty}\int^b_ae^{F(\tau)}(L(\gamma_{n_k},\dot{\gamma}_{n_k},\tau)+\alpha)\mbox{d}\tau
	\geq \int^b_ae^{F(\tau)}(L(\gamma_{x,t}^-,\dot{\gamma}_{x,t}^-,\tau)+\alpha)\mbox{d}\tau.
\end{align*}
Combining with (\ref{eq:dominant}), we get (\ref{eq:calibrated_0}). Since $\gamma_{x,t}^-|_{[s,t]}$ is a minimizer of (\ref{eq:dominant_min}) for each $s<t$, due to Lemma \ref{Sec3:minimizer}, $\gamma_{x,t}^-$ is $C^r$ and solves (\ref{eq:e-l}).\qed
\endproof
\begin{rmk}
Due to Lemma \ref{Sec3:calibrated dot bounded}, if we take $t_2=t$ and make $t_1\rightarrow-\infty$ in (\ref{eq:calibrated_0}) we instantly get 
%$\big\{\int_{-n}^te^{F(\tau)-F(t)}(L(\gamma_*(\tau),\dot{\gamma}_*(\tau),\tau)+\alpha)\mbox{d}\tau\big\}_{n\in\Z_+}$ is uniformly bounded, which instantly leads to 
%\begin{align*}
%	\int_{-n}^te^{F(\tau)-F(t)}(L(\gamma_*(\tau),\dot{\gamma}_*(\tau),\tau)+\alpha)\mbox{d}\tau&=\int_{-\infty}^te^{F(\tau)-F(t)}(L(\gamma_*(\tau),\dot{\gamma}_*(\tau),\tau)+\alpha)\chi_{[-n,t]}(\tau) \mbox{d}\tau\\
%	&\leq \int^t_{-\infty}e^{F(\tau)-F(t)}(C_{\kappa^*}+\alpha)\mbox{d}\tau\\
%	&\leq (C_{\kappa^*}+\alpha)\frac{e^{2k_0+[f]}}{[f]}.
%\end{align*}
%On the other hand,
%\begin{align*}
%	\int_{-n}^te^{F(\tau)-F(t)}(L(\gamma_*(\tau),\dot{\gamma}_*(\tau),\tau)+\alpha)\mbox{d}\tau&=\int_{-\infty}^te^{F(\tau)-F(t)}(L(\gamma_*(\tau),\dot{\gamma}_*(\tau),\tau)+\alpha)\chi_{[-n,t]}(\tau) \mbox{d}\tau\\
%	&\geq -|C(0)-\alpha|\int^t_{-\infty}e^{F(\tau)-F(t)}\mbox{d}\tau\\
%	&\geq -|C(0)-\alpha|\frac{e^{2k_0+[f]}}{[f]}.
%\end{align*}
%By Dominated Convergence Theorem, we derive 
$$
u_{\alpha}^-(x,t)=\int^t_{-\infty}e^{F(\tau)-F(t)}(L(\gamma_{x,t}^-(\tau),\dot{\gamma}_{x,t}^-(\tau),\tau)+\alpha)\mbox{d}\tau,
$$
i.e. the infimum in (\ref{Sec3:solution}) is achieved at $\gamma_{x,t}^-:(-\infty,t]\rightarrow M$.
%by the Dominated Convergence Theorem.
\end{rmk}

\begin{lem}\label{Sec3:calibrated dot bounded}
	Suppose $\gamma_{x,\theta}^-:(-\infty,\theta]\to M$ is a backward calibrated curve ending with $x$ of $u_{\alpha}^-(x,\theta)$, then
	 \[
	|\dot{\gamma}_{x,\theta}^-(\tau)|\leq \kappa_0,\quad\forall\  (x,\theta)\in M\times\T, \tau<\theta.
	\]
for a constant $\kappa_0$ depending on $L$ and $\alpha$. That implies $\gamma_{x,\theta}^-$ is actually Lipschitz on $(-\infty,\theta]$.
\end{lem}
\proof
 Let $s_1,s_2\leq \theta$ and $s_2-s_1=1$. Due to Lemma \ref{Sec3:calibrated}, 
\begin{align*}
	&e^{F(s_2)}u^-_{\alpha}(\gamma^-_{x,\theta}(s_2),s_2)-e^{F(s_1)}u^-_{\alpha}(\gamma^-_{x,\theta}(s_1),s_1)\\
	&=\int^{s_2}_{s_1}e^{F(\tau)}(L(\gamma^-_{x,\theta}(\tau),\dot{\gamma}^-_{x,\theta}(\tau),\tau)+\alpha)\mbox{d}\tau\\
	&\geq\int^{s_2}_{s_1}e^{F(\tau)}(|\dot{\gamma}^-_{x,\theta}(\tau)|-C(1)+\alpha)\mbox{d}\tau.
\end{align*}
On the other hand, let $\beta:[s_1,s_2]\to M$ be a geodesic satisfying
$\beta(s_1)=\gamma^-_{x,\theta}(s_1),\beta(s_2)=\gamma^-_{x,\theta}(s_2)$, and $|\dot\beta(\tau)|\leq \mbox{diam(M)}=:k_1$. Then
\begin{align*}
	&e^{F(s_2)}u^-_{\alpha}(\gamma^-_{x,\theta}(s_2),s_2)-e^{F(s_1)}u^-_{\alpha}(\gamma^-_{x,\theta}(s_1),s_1)\\
	&\leq \int^{s_2}_{s_1}e^{F(\tau)}(L(\beta(\tau),\dot{\beta}(\tau),\tau)+\alpha)\mbox{d}\tau\\
	&\leq \int^{s_2}_{s_1}e^{F(\tau)}(C_{k_1}+\alpha)\mbox{d}\tau.
\end{align*}
Hence,
$$
\int^{s_2}_{s_1}e^{F(\tau)}|\dot{\gamma}^-_{x,\theta}(\tau)|\mbox{d}\tau\leq \int^{s_2}_{s_1}e^{F(\tau)}(C_{k_1}+C(1))\mbox{d}\tau.
$$
Due to the continuity of $\dot{\gamma}^-_{x,\theta}(\tau)$, there exists $s_0\in(s_1,s_2)$ such that
\begin{equation}\label{eq:dot_bd}
	|\dot{\gamma}^-_{x,\theta}(s_0)|\leq C_{k_1}+C(1).
\end{equation}
%This means that for each interval with length 1, there exists a point in the interval at which the derivative of minimal curve is uniformly bounded for all $(x,t)\in M\times\mathbb{R}$. 
Note that $\gamma^-_{x,\theta}$ solves \eqref{eq:e-l}, so $|\dot{\gamma}^-_{x,\theta}(\tau)|$ is uniformly bounded for $(x,\theta)\in M\times\mathbb{T}$ and $\tau\in(-\infty,\theta]$.\qed\medskip
%{\color{blue}
%Formula (\ref{eq:dot_bd}) means that there exists a $s_0$ in each interval with length 1, such that $\dot{\gamma}_{x,\theta}^-$ is bounded by a constant depending only on $L$ and the diameter of $M$.
%}
%\textr{
%We assume $\kappa_0$ is the Lipschitz constant of backward calibrated curve of $u_{\alpha}(x,t)$. Let $p_0:\mathbb{R}\to\mathbb{T}$ be the projection. The distance in $\mathbb{T}$ is defined by
%$$
%d(p_0(t),p_0(t'))=\min_{k\in\mathbb{Z}}|t-t'+k|
%$$
%}
\begin{lem}{\bf [(2) of Theorem \ref{thm:1}]}\label{lem:lip-dis}
For each  $\alpha\in\mathbb{R}$, $u_{\alpha}^-$	 is Lipschitz on $M\times\T$.
\end{lem}
\proof 
 First of all, we prove $u_\alpha^-(\cdot,\theta):M\rightarrow\R$ is uniformly  Lipschitz w.r.t. $\theta\in\mathbb{T}$. Let $x,y\in M$, $\Delta t=d(x,y)$, and $\gamma^-_{x,\theta}:(-\infty,\theta]\to M$ be a minimizer of $u^-_{\alpha}(x,\theta)$. Define $\tilde{\gamma}:(-\infty,\theta]\to M$ by
$$
\tilde{\gamma}(s)=
\begin{cases}
	\gamma^-_{x,\theta}(s),s\in(-\infty,\theta-\Delta t),\\
	\beta(s),s\in[\theta-\Delta t,\theta],
\end{cases}
$$
where $\beta:[\theta-\Delta t,\theta]\to M$ is a geodesic satisfying $\beta(\theta-\Delta t)=\gamma_{x,\theta}^-(\theta-\Delta t),\beta(\theta)=y$, and 
$$
|\dot{\beta}(s)|\equiv \frac{d(\gamma^-_{x,\theta}(\theta-\Delta t),y)}{\Delta t}\leq \frac{d(\gamma_{x,\theta}^-(\theta-\Delta t),x)}{\Delta t}+1\leq \kappa_0+1.
$$
Then,
\begin{align*}
	&u^-_{\alpha}(x,\theta)=\int^\theta_{-\infty}e^{F(\tau)-F(\theta)}(L(\gamma^-_{x,\theta}(\tau),\dot{\gamma}^-_{x,\theta}(\tau),\tau)+\alpha)\mbox{d}\tau,\\
	&u^-_{\alpha}(y,\theta)\leq \int^{\theta}_{-\infty}e^{F(\tau)-F(\theta)}(L(\tilde{\gamma}(\tau),\dot{\tilde{\gamma}}(\tau),\tau)+\alpha)\mbox{d}\tau,
\end{align*}
which implies
\begin{align*}
	u^-_{\alpha}(y,\theta)-u^-_{\alpha}(x,\theta)&\leq \int^\theta_{\theta-\Delta t}e^{F(\tau)-F(\theta)}(L(\beta,\dot{\beta},\tau)-L(\gamma_{x,\theta}^-,\dot{\gamma}_{x,\theta}^-,\tau))\mbox{d}\tau\\	
	&\leq (C_{\kappa_0+1}+C(0))\int^\theta_{\theta-\Delta t}e^{F(\tau)-F(\theta)}\mbox{d}\tau\\
	&\leq (C_{\kappa_0+1}+C(0))e^{2k_0+[f]}\cdot d(x,y).
\end{align*}

By a similar approach, we derive the opposite inequality holds. Hence,
\begin{equation}\label{eq:xlip}
|u^-_{\alpha}(y,\theta)-u^-_{\alpha}(x,\theta)|\leq \rho_*\cdot d(x,y),	
\end{equation}
where $\rho_*=(C_{\kappa_0+1}+C(0))e^{2k_0+[f]}$.\medskip

Next, we prove $u^-_{\alpha}(x,\cdot)$ is uniformly Lipschitz continuous for $x\in M$. Let $\bar{t},\bar{t}'\in\mathbb{T}, d(\bar{t},\bar{t}')=t'-t$, and $t\in[0,1)$. Then, $t'\in[0,2]$. 
A curve $\eta:(-\infty,t']\to M$ is defined by
$$
\eta(s)=
\begin{cases}
	\gamma^-_{x,t}(s),s\in(-\infty,t],\\
	x,\ \ s\in (t,t'].
\end{cases}
$$
Then,
\begin{align*}
	&\ \ \ \ e^{F(t')}u^-_{\alpha}(x,t')-e^{F(t)}u^-_{\alpha}(x,t)\\
	&\leq\int^{t'}_{-\infty}e^{F(\tau)}(L(\eta,\dot{\eta},\tau)+\alpha)\mbox{d}\tau-\int^t_{-\infty}e^{F(\tau)}(L(\gamma^-_{x,t},\dot{\gamma}^-_{x,t},\tau)+\alpha)\mbox{d}\tau\\
	&\leq \int^{t'}_te^{F(\tau)}(C_0+\alpha)\mbox{d}\tau\\
	&\leq (C_0+\alpha)\max_{\tau\in[0,2]}e^{F(\tau)}\cdot|t'-t|.
	\end{align*}
On the other hand, 
we write $\Delta t=d(\bar{t}',\bar{t})$ and define $\eta_1\in C^{ac}((-\infty,t],M)$ by
$$
\eta_1(s)=
\begin{cases}
	\gamma^-_{x,t'}(s),s\in(-\infty,t-\Delta t],\\
	\gamma^-_{x,t'}(2(s-t)+t'),s\in(t-\Delta t,t].
\end{cases}
$$
It is easy to check $\eta_1(t)=x$, and $|\dot{\eta}_1(\tau)|\leq 2\kappa_0$, where $\kappa_0$ is a Lipschitz constant of $\gamma_{x,t'}^-$.
\begin{align*}
	e^{F(t)}u^-_{\alpha}(x,t)&\leq \int^{t}_{-\infty}e^{F(\tau)}(L(\eta_1(\tau),\dot{\eta}_1(\tau),\tau)+\alpha)\mbox{d}\tau\\
	&\leq \int^{t}_{t-\Delta t}e^{F(\tau)}(L(\eta_1(\tau),\dot{\eta}_1(\tau),\tau)+\alpha)\mbox{d}\tau\\
	&+\int^{t-\Delta t}_{-\infty}e^{F(\tau)}(L(\gamma^-_{x,t'}(\tau),\dot{\gamma}^-_{x,t'}(\tau),\tau)+\alpha)\mbox{d}\tau.
	\end{align*}
 Note that $\gamma^-_{x,t'}$ is a minimizer of $u^-_\alpha(x,t')$. We derive that
 \begin{align*}
	&\ \ \ e^{F(t)}u^-_{\alpha}(x,t)-e^{F(t')}u^-_\alpha(x,t')\\
	&\leq \int^t_{t-\Delta t}e^{F(\tau)}(L(\eta_1(\tau),\dot{\eta}_1(\tau),\tau)+\alpha)\mbox{d}\tau\\
	&-\int^{t'}_{t-\Delta t}e^{F(\tau)}(L(\gamma^-_{x,t'}(\tau),\dot{\gamma}^-_{x,t'}(\tau),\tau)+\alpha)\mbox{d}\tau\\
	&\leq (C_{2\kappa_0}+2C(0)+|\alpha|)\max_{\tau\in[0,2]} e^{F(\tau)}\cdot d(\bar{t}',\bar{t}).
\end{align*}
We have proved the map $t\longmapsto e^{F(t)}u^-_{\alpha}(x,t)$ is uniformly Lipschitz for $x\in M$, with Lipschitz constant depends only on $L,f$ and $\alpha$. 
Note that $F(t)$ is $C^{r+1}$ and $F'(t)=f(t)$ is 1-periodic. We derive $u^-_\alpha(x,\cdot)$ is uniformly Lipschitz for $x\in M$ with Lipschitz constant $\rho^*_0$ depending on $L,f$, and $\alpha$.
%Hence, $t\longmapsto e^{F(t)}u^-_{\alpha}(x,t)$ is Lipschitz continuous and
%$$
%\bigg|\frac{\partial }{\partial t}\big(e^{F(t)}u^-_{\alpha}(x,t)\big)\bigg|\leq \rho_0, a.e.t\in[0,1]. 
%$$
%Then,
%\begin{align*}
%\bigg|\frac{\partial u^-_{\alpha}(x,t)}{\partial t}\bigg|&=\bigg|\frac{\partial}{\partial t}\big(e^{F(t)}u^-_{\alpha}(x,t)\big)e^{-F(t)}-e^{-F(t)}f(t)\big(u^-_{\alpha}(x,t)e^{F(t)}\big)\bigg|\\
%&=\bigg|\frac{\partial}{\partial t}\big(e^{F(t)}u^-_{\alpha}(x,t)\big)e^{-F(t)}-f(t)u^-_{\alpha}(x,t)\bigg|\\
%&\leq\rho_0\max_{t\in [0,1]}e^{F(t)}+\max_{t\in[0,1]}|f(t)|\cdot|u^-_{\alpha}(x,t)|.
%\end{align*}
%Note that $u^-_{\alpha}(x,t)$ is bounded. We derive $t\longmapsto u^-_{\alpha}(x,t)$ is uniformly Lipschitz continuous for $x\in M$ and $t\in[0,1]$.
%We derive 
%\begin{equation}\label{eq:tlip}
%|u^-_{\alpha}(x,\bar{t}')-u^-_{\alpha}(x,\bar{t})|\leq \rho^*_0\cdot d(\bar{t}',\bar{t}).	
%\end{equation}
It follows that
\begin{align*}
|u^-_{\alpha}(x',\theta')-u^-_{\alpha}(x,\theta)|&\leq |u^-_{\alpha}(x',\theta')-u^-_{\alpha}(x,\theta')|+|u^-_{\alpha}(x,\theta')-u^-_{\alpha}(x,\theta)|\\
&\leq\rho_*d(x',x)+\rho^*_0d(\theta',\theta)
\end{align*}
so we finish the proof.
\qed
%{\color{blue}
%By Lemmas \ref{Sec3:properties} and \ref{Sec3:approximate}, for $t-t'\geq diam(M)$,
%$$
%u_{\alpha}^-(x,t)=\lim_{s\to-\infty}\mathcal{T}^{\alpha-}_{s,t}\phi(x)=\lim_{s\to-\infty}\mathcal{T}^{\alpha-}_{t',t}\circ\mathcal{T}^{\alpha-}_{s,t'}\phi(x)=\mathcal{T}^{\alpha-}_{t',t}\lim_{s\to-\infty}\mathcal{T}^{\alpha-}_{s,t'}\phi(x)=\mathcal{T}^{\alpha-}_{t',t}u_{\alpha}^-(x,t'),
%$$
%where $\phi\in C(M,\mathbb{R})$.
%Note that $u_{\alpha}^-$ is bounded with boundedness depending only on $\alpha$ and $[f]$. From Lemma \ref{Sec3:T-lip}, we derive that
%$(x,t)\longmapsto\mathcal{T}^{\alpha-}_{t',t}u^-_{\alpha}(x,t')$ is Lipschitz. Hence, the function $u_{\alpha}^-(x,t)$ is Lipschitz with Lipschitz constant depending only on $\alpha$ and $f$.
%}
%\qed
%We remark that the convergence in (\ref{eq:convergence}) is uniform on $M\times K$, for each compact set $K\subset\mathbb{R}$.
\begin{lem} {\bf [(5) of Theorem \ref{thm:1}]}\label{lem:vis-sol-1}
The function $u_{\alpha}^-(x,t)$ defined by (\ref{Sec3:solution}) is a viscosity solution of (\ref{eq:sta-hj}).
\end{lem}
\proof 
 Let $\phi^*(x,t)$ be a $C^1$ function such that $u^-_{\alpha}(x,t)-\phi^*(x,t)$ attains maximum at $(x_0,t_0)$ and $u(x_0,t_0)=\phi^*(x_0,t_0)$.
For each $v\in T_{x_0}M$, there exists a $C^1$ curve $\gamma$ defined on a neighborhood of $t_0$ with $\dot{\gamma}(t_0)=v$ and $\gamma(t_0)=x_0$.
Let $\Delta t<0$. 
Then
\ben
	& &e^{F(t_0)}\phi^*(\gamma(t_0),t_0)-e^{F(t_0)}\phi^*(\gamma(t_0+\Delta t),t_0+\Delta t)\\
	&\leq& e^{F(t_0)}u^-_{\alpha}(\gamma(t_0),t_0)-e^{F(t_0)}u^-_{\alpha}(\gamma(t_0+\Delta t),t_0+\Delta t)\\
	&=&e^{F(t_0)}u^-_{\alpha}(\gamma(t_0),t_0)-e^{F(t_0+\Delta t)}u^-_{\alpha}(\gamma(t_0+\Delta t),t_0+\Delta t)\\
	& &+(e^{F(t_0+\Delta t)}-e^{F(t_0)})u^-_{\alpha}(\gamma(t_0+\Delta t),t_0+\Delta t).
\een
%Note that
%\begin{align*}
%e^{F(t_0)}u^-_{\alpha}(\gamma(t_0),t_0)-e^{F(t_0+\Delta t)}u^-_{\alpha}(\gamma(t_0+\Delta t),t_0+\Delta t)\leq \int^{t_0}_{t_0+\Delta t}e^{F(\tau)}(L(\gamma(\tau),\dot{\gamma}(\tau),\tau)+\alpha)\mbox{d}\tau.	
%\end{align*}
By (\ref{eq:dominant}),
we derive that
\begin{align*}
	&e^{F(t_0)}\bigg(\frac{\phi^*(\gamma(t_0+\Delta t),t_0+\Delta t)-\phi^*(\gamma(t_0),t_0)}{\Delta t}\bigg)\\
	&\leq \frac{1}{\Delta t}\int^{t_0+\Delta t}_{t_0}e^{F(\tau)}(L(\gamma(\tau),\dot{\gamma}(\tau),\tau)+\alpha)\mbox{d}\tau\\
	&\ \ \ \ \ \ \  \ \ \ -\bigg(\frac{e^{F(t_0+\Delta t)}-e^{F(t_0)}}{\Delta t}\bigg)u^-_{\alpha}(\gamma(t_0+\Delta t),t_0+\Delta t).
\end{align*}
Taking $\Delta t\to 0^-$, we derive that
$$
\partial_t\phi^*(x_0,t_0)+\partial_x\phi^*(x_0,t_0)\cdot v-L(x_0,v,t_0)+f(t_0)u^-_{\alpha}(x_0,t_0)\leq \alpha.
$$
By the arbitrariness of $v$,
$$
\partial_t\phi^*(x_0,t_0)+H(x_0,\partial_x\phi^*(x_0,t_0),t_0)+f(t_0)u^-_{\alpha}(x_0,t_0)\leq\alpha,
$$
which implies $u^-_{\alpha}(x,t)$ is a viscosity subsolution of (\ref{eq:sta-hj}).

Let $(x_0,t_0)\in M\times\mathbb{R}$ and $\gamma^-_{x_0,t_0}:(-\infty,t_0]\to M$ be a minimizer of $u^-_{\alpha}(x_0,t_0)$ and let $\phi_*(x,t)\in C^1(M\times\mathbb{R},\mathbb{R})$ such that
$u^-_{\alpha}(x,t)-\phi_*(x,t)$ attains minimum at $(x_0,t_0)$. Then, for $\Delta t<0$,
\begin{align*}
	&\ \  e^{F(t_0)}(\phi_*(\gamma^-_{x_0,t_0}(t_0),t_0)-\phi_*(\gamma^-_{x_0,t_0}(t_0+\Delta t),t_0+\Delta t))\\
	&\geq e^{F(t_0)}u^-_{\alpha}(\gamma^-_{x_0,t_0}(t_0),t_0)-e^{F(t_0+\Delta t)}u^-_{\alpha}(\gamma^-_{x_0,t_0}(t_0+\Delta t),t_0+\Delta t)\\
	& \ \   +e^{F(t_0+\Delta t)}u^-_{\alpha}(\gamma^-_{x_0,t_0}(t_0+\Delta t),t_0+\Delta t)-e^{F(t_0)}u^-_{\alpha}(\gamma^-_{x_0,t_0}(t_0+\Delta t),t_0+\Delta t)\\
	&=\int^{t_0}_{t_0+\Delta t}e^{F(\tau)}(L(\gamma^-_{x_0,t_0}(\tau),\dot{\gamma}^-_{x_0,t_0}(\tau),\tau)+\alpha)\mbox{d}\tau\\
	&\ \  +(e^{F(t_0+\Delta t)}-e^{F(t_0)})u^-_{\alpha}(\gamma^-_{x_0,t_0}(t_0+\Delta t),t_0+\Delta t).
\end{align*}
Then
\begin{align*}
	&e^{F(t_0)}\frac{\phi_*(\gamma^-_{x_0,t_0}(t_0+\Delta t),t_0+\Delta t)-\phi_*(\gamma^-_{x_0,t_0}(t_0),t_0)}{\Delta t}\\
	&\geq \frac{1}{\Delta t}\int^{t_0+\Delta t}_{t_0}e^{F(\tau)}(L(\gamma^-_{x_0,t_0}(\tau),\dot{\gamma}^-_{x_0,t_0}(\tau),\tau)+\alpha)\mbox{d}\tau\\
	&-\bigg(\frac{e^{F(t_0+\Delta t)}-e^{F(t_0)}}{\Delta t}\bigg)u^-_{\alpha}(\gamma^-_{x_0,t_0}(t_0+\Delta t),t_0+\Delta t).
\end{align*}
Taking $\Delta t\to 0^-$, we derive that
\[
\partial_t\phi_*(x_0,t_0)+H(x_0,\partial_x\phi_*(x_0,t_0),t_0)+f(t_0)u(x_0,t_0)\geq \alpha
\]
which implies the assertion.
\qed\medskip

%By a similar approach in the proof of Lemma \ref{Sec3:dotbounded}, we derive the following lemma.

%\vspace{20pt}

%\noindent{\it Proof of Theorem \ref{thm:1}:} The detailed proof has been included in above analysis.\qed

%\vspace{20pt}

%\begin{prop}\cite[Proposition 7]{MS}\label{prop:ms}
%For system (\ref{eq:ham-main}) satisfying $(H1-3)+(H0^-)$ and any $\alpha\in\R$, the weak KAM solution $u^-$ of (\ref{eq:sta-hj}) satisfies:
%For every $(x,s)\in M\times\T$, there exist a backward calibrated curve 
%\[
%\wt\gamma_{x,s}(t)=\begin{pmatrix}
%      \gamma_{x,s}^- (t)   \\
%       t+s
%\end{pmatrix}:t\in (-\infty,0]\rightarrow M\times\T
%\]
%ending with it, along which $u^-$ is differentiable for all $t<0$. Precisely, we have
%\be
%& &f(t+s) u^-(\gamma_{x,s}^-(t),t+s)+\partial_t u^-(\gamma_{x,s}^-(t),t+s)\nonumber\\
%&+&H(\gamma^-_{x,s}(t),\partial_xu^-(\gamma_{x,s}^-(t),t+s),t+s)=\alpha,\quad\forall t<0.
%\ee
%and
%\be
%(\gamma_{x,s}^-(t),\dot\gamma_{x,s}^-(t))=\cL\Big(\gamma_{x,s}^-(t),\partial_xu^-(\gamma_{x,s}^-(t),t+s)\Big),\quad\forall t<0.
%\ee
%\end{prop}
As an complement, the following result is analogue to Proposition 6 of \cite{MS} will be useful in the following sections:
\begin{prop}
%{\bf [(6) of Theorem \ref{thm:1}]}, Proposition 6 of \cite{MS}
\label{Sec3:pro_differentiable}
	%For system (\ref{eq:ham-main}) satisfying {\bf (H1-3)+(H0$^-$)}, $\alpha\in\R$ and $(x,t)\in M\times\mathbb{T}$, 
	The weak KAM solution $u_\alpha^-$ of (\ref{eq:sta-hj}) is differentiable at $(\gamma_{x,t}^-(s),\bar s)$ for any $\R\ni s<t$, where $\gamma_{x,t}^-: (-\infty,t]\to M$ is a  backward calibrated curve 
	ending with $x$. In other words, we have 
	$$
	\partial_tu^-(\gamma_{x,t}^-(s),s)+H(\gamma_{x,t}^-(s),\partial_xu^-_\alpha(\gamma_{x,t}^-(s),s),s)+f(s)u^-_\alpha(\gamma_{x,t}^-(s),s)=\alpha
	$$	
and 
\be
(\gamma_{x,t}^-(s),\dot\gamma_{x,t}^-(s),\bar s)=\cL\Big(\gamma_{x,t}^-(s),\partial_xu^-_\alpha(\gamma_{x,t}^-(s),s),\bar s\Big)
\ee
for all $\R\ni s<t$.
\end{prop}
\proof	
%Without loss generality, we prove the proposition when $M=\mathbb{R}^n$.
By Theorem \ref{Sec3:thm_semiconcave2}, we derive $u^-_\alpha(x,t)$ is semiconcave. Let $s\in(-\infty,t)$ and $\tilde{p}=(p_x,p_t)\in D^+u^-_{\alpha}(\gamma_{x,t}^-(s),s)$.
For $\Delta s>0$, 
\ben
	& &\frac{e^{F(s+\Delta s)}u^-_\alpha(\gamma_{x,t}^-(s+\Delta s),s+\Delta s)-e^{F(s)}u^-_{\alpha}(\gamma^-_{x,t}(s),s)}{\Delta s}\\
	&=&\frac{1}{\Delta s}\int^{s+\Delta s}_{s}e^{F(\tau)}(L(\gamma_{x,t}^-(\tau),\dot{\gamma}_{x,t}^-(\tau),\tau)+\alpha)\mbox{d}\tau.
\een
Then
\ben
	& &\lim_{\Delta s\to 0^+}\frac{u^-_{\alpha}(\gamma_{x,t}^-(s+\Delta s),s+\Delta s)-u^-_{\alpha}(\gamma_{x,t}^-(s),s)}{\Delta s}\\
	&=&L(\gamma_{x,t}^-(s),\dot{\gamma}_{x,t}^-(s),s)+\alpha-f(s)u^-_{\alpha}(\gamma_{x,t}^-(s),s).
\een
By Proposition \ref{Sec3:prop_semiconcave1}, 
$$
\lim_{\Delta s\to 0^+}\frac{u^-_{\alpha}(\gamma_{x,t}^-(s+\Delta s),s+\Delta s)-u^-_{\alpha}(\gamma_{x,t}^-(s),s)}{\Delta s}\leq p_x\cdot\dot{\gamma}_{x,t}^-(s)+p_t,
$$
which implies
\ben
& &p_t+H(\gamma_{x,t}^-(s),p_x,s)+f(s)u^-_{\alpha}(\gamma_{x,t}^-(s),s)\geq \alpha.
\een
On the other hand, $u^-_\alpha$ is a viscosity solution of (\ref{eq:sta-hj}). Hence, for each $(p_x,p_t)\in D^+u^-_{\alpha}(\gamma_{x,t}^-(s),s)$,
\begin{equation}\label{eq:singleton}
	p_t+H(\gamma_{x,t}^-(s),p_x,s)+f(s)u^-_{\alpha}(\gamma_{x,t}^-(s),s)=\alpha.
\end{equation}
Note that $H(x,p,t)$ is strictly convex with respect to $p$. By (\ref{eq:singleton}), We derive $D^+u^-_{\alpha}(\gamma_{x,t}^-(s),s)$ is a singleton.
 By Proposition \ref{Sec3:prop_semiconcave1}, $u^-_{\alpha}(x,t)$ is differentiable at $(\gamma_{x,t}^-(s),s)$.\qed

%\begin{itemize}
%%\item $u^-$ is Lipschitz on $M\times\T$, with the Lipschitz constant depending on $L(x,v,t)$, $f(t)$ and $\alpha$.
%\item For every $(x,s)\in M\times\T$, there exist a backward calibrated curve 
%\[
%\wt\gamma_{x,s}(t)=\begin{pmatrix}
%      \gamma_{x,s}^- (t)   \\
%       t+s
%\end{pmatrix}:t\in (-\infty,0]\rightarrow M\times\T
%\]
%ending with it, along which $u^-$ is differentiable for all $t<0$. Precisely, we have
%\be
%& &f(t+s) u^-(\gamma_{x,s}^-(t),t+s)+\partial_t u^-(\gamma_{x,s}^-(t),t+s)\nonumber\\
%&+&H(\gamma^-_{x,s}(t),\partial_xu^-(\gamma_{x,s}^-(t),t+s),t+s)=\alpha,\quad\forall t<0.
%\ee
%and
%\be
%(\gamma_{x,s}^-(t),\dot\gamma_{x,s}^-(t))=\cL\Big(\gamma_{x,s}^-(t),\partial_xu^-(\gamma_{x,s}^-(t),t+s)\Big),\quad\forall t<0.
%\ee
%%\item There exists a constant $A>0$ depending on $L(x,v,t)$, $f(t)$ and $\alpha$, such that for all $(x,s)\in M\times\T$, $\|\dot\gamma_{x,s}^-(t)\|\leq A$ for all $t\leq 0$.
%%\textr{\item For all $(x,s)\in M\times\T$, 
%\[
%\cL^{-1}(\gamma_{x,s}^-(0),\dot\gamma_{x,s}^-(0))\in \big\{(x,D^+u^-(x,s))\in TM\big\}
%\]
%where $D^+u^-(x,s)$ is the {\sf sub-derivative} of $u^-$ w.r.t. $x-$variable. 
%}
%\end{itemize}

\subsection{Weak KAM solution of (\ref{eq:sta-hj2}) in the condition (\textbf{H0}$^0$)}

Now $[f]=0$, so $F(t):=\int^t_0 f(\tau)\mbox{d}\tau$ is 1-periodic. Let $\ell=\int^1_0e^{F(\tau)}\mbox{d}\tau$, then we define a new Lagrangian $\mathbf{L}:TM\times\mathbb{T}\to\mathbb{R}$ by
$$
\mathbf{L}(x,v,t)=e^{F(t)}L(x,v,t).
$$
For such a  $\mathbf{L}$, {\sf the Peierls Barrier} $\textbf{h}^\infty_\alpha:M\times\mathbb{T}\times M\times\mathbb{T}\to\mathbb{R}$ 
$$
\textbf{h}^\infty_\alpha(x,\bar{s},y,\bar{t})=\liminf_{\substack{t\equiv\bar{t},s\equiv \bar{s} (mod\; 1)\\t-s\to+\infty}}\inf_{\substack{\gamma\in C^{ac}([s,t],M) \\ \gamma(s)=x,\gamma(t)=y }}\int^t_s\mathbf{L}(\gamma,\dot{\gamma},\tau)+\alpha\cdot\ell \mbox{d}\tau,
$$
is well-defined, once $\alpha$ is uniquely established by 
\begin{equation}\label{eq:def_critical}
c(H)=\inf\{\alpha\in\R|\int^t_s\mathbf{L}(\gamma,\dot{\gamma},\tau)+\alpha\cdot\ell\mbox{d}\tau\geq 0,\ \forall \gamma\in\mathcal{C} \}
\end{equation}
with $\mathcal{C}=\{\gamma\in C^{ac}([s,t],M)|\gamma(s)=\gamma(t)\mbox{ and }t-s\in\mathbb{Z}_+\}$, due to Proposition 2 of \cite{CIM}. Moreover,
%$p_0:\mathbb{R}\to \mathbb{T}$ is the projection.
%The critical value $c(H)$ is defined by
%$$
%c(H)=\inf\{\alpha|\int^t_s\widetilde{L}(\gamma,\dot{\gamma},\tau)+\alpha\cdot\ell\mbox{d}\tau\geq 0,\ \forall \gamma\in\mathcal{C} \},
%$$
the following properties were proved in \cite{CIM}:
\begin{prop}\label{pro3:basic property}
\begin{description}
	\item [(i)] If $\alpha<c(H)$, $\textbf{h}^\infty_\alpha\equiv-\infty$.
		\item [(ii)] If $\alpha>c(H)$, $\textbf{h}^\infty_\alpha\equiv+\infty$.
		\item [(iii)] $\textbf{h}^\infty_{c(H)}$ is finite.
		\item [(iv)] $\textbf{h}^\infty_{c(H)}$ is Lipschitz.
		\item [(v)] For each $\gamma\in C^{ac}([s,t],M)$ with $\gamma(s)=x,\gamma(t)=y$,
		$$\textbf{h}^\infty_{c(H)}(z,\bar{\varsigma},y,\bar{t})-\textbf{h}^\infty_{c(H)}(z,\bar{\varsigma},x,\bar{s})\leq\int^t_s\mathbf{L}(\gamma,\dot{\gamma},\tau)+c(H)\cdot\ell\mbox{d}\tau.
		$$
\end{description}
\end{prop}

Consequently, for any $(z,\bar{\varsigma})\in M\times\mathbb{T}$ fixed, we construct a function $u_{z,\bar{\varsigma}}:M\times\mathbb{T}\rightarrow \mathbb{R}$ by
\begin{equation}\label{eq3:def_u}
	u^-_{z,\bar{\varsigma}}(x,\bar{t})=e^{-F(\bar{t})}\bigg(\textbf{h}^\infty_{c(H)}(z,\bar{\varsigma},x,\bar{t})+c(H)\cdot\int^t_\varsigma e^{F(\tau)}-\ell\mbox{d}\tau\bigg).
\end{equation}

\textbf{Proof of Theorem \ref{cor:1}: }
(1) 
Due to (iv) of Proposition \ref{pro3:basic property}, $u^-_{z,\bar{\varsigma}}$ is also Lipschitz.

(2) The domination property of $u^-_{z,\bar{\varsigma}}$ can be achieved immediately by (v) of Proposition \ref{pro3:basic property}.
 
 (3) By Tonelli Theorem and the definition of $u^-_{z,\bar{\varsigma}}$, there exists a sequence ${\varsigma_{k}}$ tending to $-\infty$ and a sequence $\gamma_k\in C^{ac}([\varsigma_k,\theta],M)$ with $\gamma_k(\varsigma_k)=z,\gamma_k(\theta)=x$, such that $\gamma_k$ minimizes the action function
 
 $$
\mathcal{F}(\beta)= \inf_{\substack{\beta\in C^{ac}([\varsigma_k,\theta]) \\\beta(\varsigma_k)=z,\beta(\theta)=x }}\int^\theta_{\varsigma_k}e^{F(\tau)}(L(\beta,\dot{\beta},\tau)+c(H))\mbox{d}\tau
 $$
 and
 $$
e^{F(\theta)}u^-_{z,\bar{\varsigma}}(x,\theta)=\lim_{k\to+\infty}\int^\theta_{\varsigma_k}e^{F(\tau)}(L(\gamma_k,\dot{\gamma}_k,\tau)+c(H))\mbox{d}\tau.
 $$
 Since each $\gamma_k$ solves (\ref{eq:e-l}), which implies $\gamma_k$ is $C^r$. By a standard way, there exists $\kappa_0$ independent of the choice of $k$, such that $|\dot{\gamma}_k|\leq \kappa_0$, when $\theta-\varsigma_k\geq 1$. 
 By Ascoli Theorem, there exists a subsequence of $\{\gamma_k\}$ (denoted still by $\gamma_k$) and an absolutely continuous curve $\gamma^-_{x,\theta}:(-\infty,\theta]\to M$ such that $\gamma_k$ converges uniformly to $\gamma^-_{x,\theta}$ on each compact subset of $(-\infty,\theta]$ and $\gamma^-_{x,\theta}(\theta)=x$. Then, for each $s<\theta$,
\ben
	e^{F(\theta)}u^-_{z,\bar{\varsigma}}(x,\theta)
	&=&\lim_{k\to+\infty}\bigg(\int^s_{\varsigma_k}e^{F(\tau)}(L(\gamma_k,\dot{\gamma}_k,\tau)+c(H))\mbox{d}\tau\\
	& &+\int^{\theta}_se^{F(\tau)}(L(\gamma_k,\dot{\gamma}_k,\tau)+c(H))\mbox{d}\tau\bigg)\\
	&\geq&\liminf_{k\to+\infty}\int^s_{\varsigma_k}e^{F(\tau)}(L(\gamma_k,\dot{\gamma}_k,\tau)+c(H))\mbox{d}\tau\\
	& &+\liminf_{k\to+\infty}\int^\theta_se^{F(\tau)}(L(\gamma_k,\dot{\gamma}_k,\tau)+c(H))\mbox{d}\tau\\
	&\geq& e^{F(s)}u^-_{z,\bar{\varsigma}}(\gamma^-_{x,\theta}(s),s)+\int^\theta_se^{F(\tau)}(L(\gamma^-_{x,\theta},\dot{\gamma}^-_{x,\theta},\tau)+c(H))\mbox{d}\tau
\een
which implies $\gamma^-_{x,\theta}$ is a calibrated curve by $u^-_{z,\bar{\varsigma}}$. 

(4) By a similar approach of the proof of Lemma \ref{lem:vis-sol-1}, we derive $u^-_{z,\bar{\varsigma}}$ is also a viscosity solution of (\ref{eq:sta-hj2}).
\qed

\section{Various properties of variational invariant sets}\label{s3}

\subsection{Aubry set in the condition \textbf{(H0$^-$)}} Due to Theorem \ref{thm:1} and Proposition \ref{Sec3:pro_differentiable}, for any $(x,\bar{s})\in M\times\T$ we can find a backward calibrated curve 
\be
\wt\gamma_{x,s}^-:=\begin{pmatrix}
      \gamma_{x,s}^-(t)   \\
      \bar t  
\end{pmatrix}:t\in (-\infty,s]\rightarrow M\times\T
\ee
 ending with it, such that the associated backward orbit $\varphi_L^{t-s}(\gamma_{x,s}^-(s),\dot \gamma_{x,s}^-(s), s)$ has an $\alpha-$limit set $\wt \cA_{x,s}\subset TM\times\T$, which is invariant and graphic over $\cA_{x,s}:=\pi\wt\cA_{x,s}$. Therefore, any critical curve $\wt\gamma_{x,s}^\infty$ in $\cA_{x,s}$ has to be a globally calibrated curve, namely
\[
\wt\cA_{x,s}\subset\wt\cA,\quad(\text{resp. } \cA_{x,s}\subset\cA).
\]
So $\wt\cA\neq \emptyset$.

Recall that any critical curve in $\cA$ is globally calibrated, then due to Proposition \ref{Sec3:pro_differentiable},  that implies for any $(x,\bar{s})\in\cA$, the critical curve $\wt\gamma_{x,s}$ passing it is unique. In other words, $\pi^{-1}:\cA\rightarrow\wt\cA$ is a graph,   and
\[
\dot\gamma_{x,s}(t)=\partial_p H(\mbox{d}u^-(\gamma_{x,s}(t),t),t),\quad \forall\ t\in\R.
\]
 That indicates that $\mbox{d}u^-:\cA\rightarrow TM$ coincides with $\partial_v L\circ (\pi|_{\wt\cA})^{-1}$. 
On the other side, 
$\|\dot{\wt\gamma}_{x,s}(t)\|\leq A<+\infty$ for all $t\in\R$ due to Lemma \ref{Sec3:calibrated dot bounded}, so $\partial_v L\circ (\pi|_{\wt\cA})^{-1}$ has to be Lipschitz. So 
%we get 
%\[
%\wt\cA=\Big\{(x,\partial_p H(x,du^-(x,t),t),t)\Big|(x,t)\in\cA\Big\}
%\]
%can be easily proved. So 
$\wt\cA$ is Lipschitz over $\cA$. This is an analogue of Theorem 4.11.5 of \cite{Fa} and a.4) of \cite{MS}, which is known as {\sf Mather's graph theorem} in more earlier works \cite{Mat} for conservative Hamiltonian systems.

\begin{lem} $\wt\cA$ has an equivalent expression
\be\label{eq:mane-equi}
\wt\cA:=\{(\gamma(t),\dot\gamma(t),\bar t)\in TM\times\T|\;\forall\; a<b\in\R, \gamma \text{ achieves $h_{\alpha}^{a,b}(\gamma(a),\gamma(b))$}\}.\quad
\ee
\end{lem}
\proof
Let $\gamma:\mathbb{R}\to M$ be a globally calibrated curve by $u^-_\alpha$.
Due to (3) and (4) of Theorem \ref{thm:1}, for $a<b\in\mathbb{R}$,
\begin{align*}
	\int^b_ae^{F(\tau)}(L(\gamma,\dot{\gamma},\tau)+\alpha)\mbox{d}\tau&=
e^{F(b)}u_\alpha^-(\gamma(b),b)-e^{F(a)}u_\alpha^-(\gamma(a),a)\\
&\leq h^{a,b}_\alpha(\gamma(a),\gamma(b)).
\end{align*}
Due to the definition of $h^{a,b}_\alpha(\gamma(a),\gamma(b))$, 
we derive $\gamma$ achieves $h^{a,b}_\alpha(\gamma(a),\gamma(b))$ for all $a<b\in\mathbb{R}$.

To prove the lemma, it suffices to show any curve $\gamma:\mathbb{R}\to M$ achieving $h^{a,b}_\alpha(\gamma(a),\gamma(b))$ for all $a<b\in\mathbb{R}$ is a calibrated curve by $u^-_\alpha$.
We claim
\begin{equation}\label{eq:mincal}
	\lim_{s\to-\infty}h^{s,t}_\alpha(z,x)=e^{F(t)}u^-_\alpha(x,t), \ \ \forall x,z\in M,t\in\mathbb{R}.
\end{equation}
Due to (3) of Theorem \ref{thm:1}, for $s<t$,
$$
e^{F(t)}u_\alpha^-(x,t)-h^{s,t}_\alpha(z,x)\leq e^{F(s)}u_\alpha^-(z,s)\to 0, s\to-\infty.
$$
On the other hand, we assume $\gamma_{x,t}$ is a globally calibrated curve by $u^-_\alpha$  with $\gamma_{x,t}(t)=x$ and $s+1<t$.
Let $\beta:[s,s+1]\to M$ be a geodesic with $\beta(s)=z,\beta(s+1)=\gamma_{x,t}(s+1)$ satisfying $|\dot{\beta}|\leq k_1:=\mbox{diam}(M)$.
Then,
\begin{align*}
	h^{s,t}_\alpha(z,x)&\leq\int^{s+1}_se^{F(\tau)}(L(\beta,\dot{\beta},\tau)+\alpha)\mbox{d}\tau+\int^t_{s+1}e^{F(\tau)}(L(\gamma_{x,t},\dot{\gamma}_{x,t},\tau)+\alpha)\mbox{d}\tau\\
	&\leq (C_{k_1}+\alpha)e^{\max f+[f][s]}+e^{F(t)}u^-_\alpha(x,t)-e^{F(s+1)}u^-_\alpha(\gamma_{x,t}(s+1),s+1).
\end{align*}
Hence,
$$
h^{s,t}_\alpha(z,x)-e^{F(t)}u^-_\alpha(x,t)\leq (C_{k_1}+\alpha)e^{\max f+[f][s]}-e^{F(s+1)}u^-_\alpha(\gamma_{x,t}(s+1),s+1).
$$
From $[f]>0$, it follows that the right side of the inequality above tending to $0$, as $s\to -\infty$.
Hence, (\ref{eq:mincal}) holds. Actually, the limit in (\ref{eq:mincal}) is uniform for $x,z\in M$ and $t\in\mathbb{R}$.

If $\gamma$ achieves $h^{a,b}_\alpha(\gamma(a),\gamma(b))$ for $a<b\in\mathbb{R}$, then
$$
h^{s,b}_\alpha(\gamma(s),\gamma(b))-h^{s,a}_\alpha(\gamma(s),\gamma(a))=\int^b_ae^{F(\tau)}(L(\gamma,\dot{\gamma},\tau)+\alpha)\mbox{d}\tau,\forall s<a.
$$
Taking $s\to-\infty$, we derive
$\gamma$ is also a calibrated curve by $u^-_\alpha$.
\qed\medskip

With the help of \eqref{eq:mane-equi}, the following Lemma can be proved:
\begin{lem}[Upper Semi-continuity]\label{lem:semi-con}
The set valued function 
\[
L\in \underbrace{C^{r\geq 2}(TM\times\T,\R)}_{\|\cdot\|_{C^r}}\longrightarrow \wt\cA\subset \underbrace{TM\times\T}_{d_{\cH}(\cdot,\cdot)}
\]
is upper semi-continuous. Here  $\|\cdot\|_{C^r}$ is the $C^r-$norm and $d_{\cH}$ is the {\sf Hausdorff distance}.
\end{lem}
\proof
It suffices to prove that for any $L_n\rightarrow L$ w.r.t. $\|\cdot\|_{C^r}-$norm, the accumulating curve of any sequence of curves $\wt\gamma_n$ in $\cA(L_n)$ should lie in $\cA(L)$.
Due to Lemma \ref{Sec3:calibrated dot bounded}, for any $n\in\Z_+$ such that $\|L_n-L\|_{C^r}\leq 1$, $\wt\cA(L_n)$ is uniformly compact in the phase space. Therefore, for any sequence $\{\wt\gamma_n\}$ each of which is globally minimal, the accumulating curve $\wt\gamma_*$ satisfies
\ben
\int_t^se^{F(\tau)}\big(L(\gamma_*,\dot\gamma_*,\tau)+\alpha\big)\mbox{d}\tau&\leq&\lim_{n\rightarrow+\infty}\int_t^se^{F(\tau)}\big(L_n(\gamma_n,\dot\gamma_n,\tau)+\alpha\big)\mbox{d}\tau\\
&\leq &\lim_{n\rightarrow+\infty}\int_t^se^{F(\tau)}\big(L_n(\eta_n,\dot\eta_n,\tau)+\alpha\big)\mbox{d}\tau
\een
for any Lipschitz continuous $\eta_n:[t,s]\rightarrow M$ ending with $\gamma_n(t)$ and $\gamma_n(s)$. Since for any Lipschitz continuous $\eta:[t,s]\rightarrow M$ ending with $\gamma_*(t)$ and $\gamma_*(s)$, we can find such a sequence $\eta_n:[t,s]\rightarrow M$ converging to $\eta$ uniformly, then we get 
\[
\int_t^se^{F(\tau)}(L(\gamma_*,\dot\gamma_*,\tau)+\alpha)\mbox{d}\tau\leq \inf_{\substack{\eta\in C^{ac}([t,s],M)\\\eta(t)=\gamma_*(t)\\\eta(s)=\gamma_*(s)}}\int_t^se^{F(\tau)}(L(\eta, \dot\eta,\tau)+\alpha)\mbox{d}\tau
\]
for any $t<s\in\R$, which implies $\gamma_*$ satisfies the Euler-Lagrange equation. Due to Theorem \ref{thm:1}, the weak KAM solution $u_*^-$ associated with $L$ is unique, so $\gamma_*$ is globally minimal, then globally calibrated by $u_*^-$, i.e. $\wt\gamma_*\in\cA(L)$.\qed 
% For any $t<s\in\R$, we have
%\[
%e^{\lb s}u_n^-(\gamma_n(s))-e^{\lb t}u_n^-(\gamma_n(t))=\int_t^se^{F(\tau)}L_n(\gamma_n,\dot\gamma_n,\tau)d\tau
%\]
%then the limit
%\[
%e^{\lb s}u_*^-(\gamma_*(s))-e^{\lb t}u_*^-(\gamma_*(t))=\int_t^se^{F(\tau)}L(\gamma_*,\dot\gamma_*,\tau)d\tau
%\]
%implies $\gamma_*$ is globally calibrated, i.e. $\gamma_*\in\cA(L)$. 

\subsection{Mather set in the condition (\textbf{H0}$^-$)} For any globally calibrated curve $\wt\gamma$, we can always find a sequence $T_n>0$, such that a $\varphi_L^t-$invariant measure $\wt\mu$ can be found by
\[
\int_{TM\times\T}f(x,v,t)\mbox{d}\wt\mu=\lim_{n\rightarrow+\infty}\frac{1}{T_n}\int_0^{T_n}f(\gamma,\dot\gamma, t)\mbox{d}t,\quad\forall f\in C_c(TM\times\T,\R).
\]
So the set of $\varphi_L^t-$invariant measures ${\mathfrak M}_L$ is not empty. 
\begin{prop}\label{prop:mat}
For all $\wt\nu\in{\mathfrak M}_L$ and $\alpha\in\R$, we have
\[
\int_{TM\times\T}L+\alpha-f(t)u_\alpha^-d\wt\nu\geq 0.
\]
Besides, 
\[
\int_{TM\times\T}L+\alpha-f(t)u_\alpha^-d \wt\nu= 0 \quad\Longleftrightarrow\quad \text{supp}(\wt\nu)\subset\wt\cA
\]
\end{prop}
\begin{proof} For any Euler-Lagrange curve $\gamma:\R\rightarrow M$ contained in $\pi_x\footnote{Here $\pi_x,\pi_t,\pi_u$ is the standard projection to the space $M,\T,\R$ respectively.} \text{supp}(\wt\nu)$, we have 
\ben
& &\int_{TM\times\T} f(t) u_\alpha^-(x,t)\mbox{d}\wt\nu\\
&=&\lim_{T\rightarrow+\infty}\frac1T\int_0^Tf(t)u_\alpha^-(\gamma(t),t)\mbox{d}t\\
&\leq& \lim_{T\rightarrow+\infty}\frac1T\int_0^Tf(t)\int_{-\infty}^t e^{F(s)-F(t)} [L(\gamma(s),\dot\gamma(s),s)+\alpha] \mbox{d}s \mbox{d}t\\
&=&\lim_{T\rightarrow+\infty}\frac1T\int_0^Tf(t)e^{-F(t)}\int_{-\infty}^te^{F(s)} [L(\gamma(s),\dot\gamma(s),s) +\alpha]\mbox{d}s\mbox{d}t\\
&=&\lim_{T\rightarrow+\infty}-\frac1T\int_0^T\Big( \int_{-\infty}^te^{F(s)} [L(\gamma(s),\dot\gamma(s),s)+\alpha] \mbox{d}s\Big)\mbox{d} e^{-F(t)}\\
&=&\lim_{T\rightarrow+\infty}-\frac1T\Big(e^{-F(t)}\int_{-\infty}^te^{F(s)} [L(\gamma(s),\dot\gamma(s),s) +\alpha]ds\Big|_0^T\Big)\\
& &+\lim_{T\rightarrow+\infty}\frac1T\int_0^TL(\gamma(t),\dot\gamma(t),t)+\alpha \mbox{d}t\\
&=&\int_{TM\times\T}L(x,v,t)+\alpha \mbox{d} \wt\nu,\nonumber
\een
which is an equality only when $\gamma$ is a backward calibrated curve of $(-\infty,t]$ for all $t\in\R$, which implies $\gamma$ is globally calibrated.\qed
\end{proof}

Due to this Proposition we can easily show that $\emptyset\neq \wt\cM\subset\wt\cA$. Moreover, as we did for the Aubry set, we can similarly get that $\pi^{-1}:\cM\rightarrow\wt\cM$ is a Lipschitz function.

\subsection{Maximal global attractor in the condition (\textbf{H0}$^-$)}

Since now $[f]>0$ and 
$
\dfrac d{dt}\wh H(x,p,\bar{s},I,u)=-f(t)\wh H(x,p,\bar{s},I,u)
$
due to Remark \ref{rmk:pro},
%by assuming the Hamiltonian $H(x,p,t)-\alpha$ in (\ref{eq:dis}). 
so for any initial point $(x,p,s,I,u)$, the $\om-$limit of trajectory $\wh \varphi_{ H}^t(x,p,\bar{s},I,u)$ lies in 
\beq
\wh\Sigma_{ H}:=\{{\wh H}(x,p,\bar{s},I,u)=0\}\subset T^*M\times T^*\T\times\R.
\eeq
\begin{lem}\label{lem:layer}
For any point $Z:=\big(x,p,\bar s,\alpha-f(s)u-H(x,p,s),u\big)\in\wh\Sigma_{ H}$ with $u\leq u^-_\alpha(x,s)$, if 
\[
\liminf_{t\rightarrow-\infty}e^{F(t)}\big|\pi_u\wh \varphi_{ H}^t(Z)\big|=0,
\]
then $\pi_x\wh \varphi_{ H}^t(Z)$ is a backward calibrated curve for $t\leq 0$. 
\end{lem}
\proof
From the equation $\dot u=\langle H_p,p\rangle-H+\alpha-f(t)u$, we derive 
\ben
e^{F(s)}\pi_uZ&=&\int_{-\infty}^0 \frac{d}{dt}e^{F(t+s)}\pi_u\wh \varphi_{ H}^t(Z)\mbox{d}t\\
&=&\int_{-\infty}^s e^{F(t)}\big(L(\cL( \varphi_{ H}^{t-s}(x,p,\bar s)))+\alpha\big)\mbox{d}t\leq u_\alpha^-(x,s),
\een
then due to the expression of $u^-_\alpha$ in (\ref{Sec3:solution}), $\pi_x\wh \varphi_{ H}^t(Z)$ is a backward calibrated curve for $t\leq 0$.\qed

This Lemma inspires us to decompose $\wh\Sigma_{ H}$ further:
\[
\left\{
\begin{aligned}
\wh\Sigma_{ H}^-:= \big\{(x,p,\bar{s},\alpha-f(s)u-H(x,p,s), u)\big| u> u^-_\alpha(x,s)\big\},\\
\wh\Sigma_{ H}^0:=\big\{(x,p,\bar{s},\alpha-f(s)u-H(x,p,s), u)\big| u= u^-_\alpha(x,s)\big\},\\
\wh\Sigma_{ H}^+:=\big\{(x,p,\bar{s},\alpha-f(s)u-H(x,p,s), u)\big| u< u^-_\alpha(x,s)\big\}.
\end{aligned}
\right.\]
\begin{lem}
For any $Z=\big(x,p,\bar s,\alpha-f(s)u-H(x,p,s),u\big)\in \wh\Sigma_{ H}$, we  have
\be\label{eq:+0}
& &\partial_t^+\Big(u_\alpha^-(\pi_{x,t}\wh \varphi_{ H}^t(Z))-\pi_u\wh \varphi_{ H}^t(Z)\Big)\\
&\leq& -f(t+s)\Big(u_\alpha^-(\pi_{x,t}\wh \varphi_{ H}^t(Z))-\pi_u\wh \varphi_{ H}^t(Z)\Big).\nonumber
\ee
Consequently, $\lim_{t\rightarrow+\infty}\wh \varphi_{H}^t(Z)\in \wh\Sigma_{ H}^-\cup\wh\Sigma_{ H}^0$.
\end{lem}
\proof
As $\wh \varphi_{ H}^t(Z)=\big(x(t),p(t),\overline{t+s},-f(s+t)u(t)-H(x(t),p(t),s+t),u(t)\big)$, then 
\ben
& &\partial_t^+\big[u_\alpha^-(x(t),s+t)-u(t)\big]\\
&\leq&\max\big\langle \partial_x^* u_\alpha^-(x(t),s+t),\dot x(t)\big\rangle+\partial_t^*u_\alpha^-(x(t),s+t)-\dot u(t)\\
&\leq& \max H(x(t),\partial_x^* u_\alpha^-(x(t),s+t),s+t)+L(x(t),\dot x(t),s+t)\\
& &+\partial^*_tu_\alpha^-(x(t),s+t)-\langle H_p(x(t),p(t),t+s),p(t)\rangle\\
& &+f(t+s)u(t)+H(x(t),p(t),s+t)-\alpha\\
&=&\max H(x(t),\partial_x^* u_\alpha^-(x(t),s+t),s+t)+\partial^*_tu^-_\alpha(x(t),s+t)\\
& &+f(t+s)u(t)-\alpha\\
&\leq &f(t+s)[u(t)-u_\alpha^-(x(t),t+s)]
\een
where  the `max' is about all the element $(\partial_x^* u_\alpha^-(x(t),s+t), \partial_t^* u_\alpha^-(x(t),s+t))$ in $D^*u_\alpha^-(x(t),s+t)$ (see Theorem \ref{thm:reachable deri} for the definition). So $\lim_{t\rightarrow+\infty}\wh \varphi_{ H}^{t}(Z)\in\wh\Sigma_{ H}^-\cup\wh\Sigma_{ H}^0$.\qed

\begin{prop}
$\Om:=\bigcap_{t\geq 0} \wh \varphi_{ H}^t(\wh\Sigma_{ H}^-\cup\wh\Sigma_{ H}^0)$ is the maximal invariant set contained in $\wh \Sigma_{H}^-\cup\wh \Sigma_{ H}^0$.
\end{prop}
\begin{proof}
Due to (\ref{eq:+0}), $\wh \Sigma_{ H}^-\cup\wh\Sigma_{ H}^0$ is forward invariant. Besides, any invariant set in $\wh\Sigma_{ H}$ has to lie in $\wh\Sigma_{ H}^-\cup\wh\Sigma_{ H}^0$. So 
$
\Om
$
is the maximal invariant set in $\wh\Sigma_{ H}^-\cup\wh\Sigma_{ H}^0$.\qed
\end{proof}

\begin{lem}
If the $p-$component of $\Om$ is bounded, then the ${u,I}-$components of $\Om$ are  also bounded.
\end{lem}
\begin{proof}
It suffices to prove that for any $(x_0,p_0,\bar{t}_0,I_0,u_0)\in T^*M\times T^*\T\times\R$, there exists a time $T(x_0,p_0,\bar{t}_0,I_0,u_0)>0$ such that for any $t\geq T$, 
\beq\label{eq:*}
\big\|\pi_{u,I}\wh\varphi_{ H}^t(x_0,p_0,\bar{t}_0,I_0,u_0)\big\|\leq C \tag{*}
\eeq
for a uniform constant $C=C(\pi_{p}\Om)$. Since $\pi_{p}\Om$ is bounded, due to the definition of $\Om$, for any  $(x_0,p_0,\bar{t}_0,I_0,u_0)\in T^*M\times T^*\T\times\R$, there always exists a time $T'(x_0,p_0,\bar{t}_0,I_0,u_0)>0$ such that for any $t\geq T'$, 
\[
\big\|\pi_p\wh \varphi_{ H}^t(x_0,p_0,\bar{t}_0,I_0,u_0)\big\|\leq C'=\frac32 \text{diam}(\pi_{p}\Om).
\]
On the other side, the $u-$equation of (\ref{eq:dis}) implies that for any $t> 0$,
\ben
& &\big\|\pi_u\wh \varphi_{ H}^{t+T'}(x_0,p_0,\bar{t}_0,I_0,u_0)\big\|\\
&\leq& e^{F(t_0+T')-F(t+T'+t_0)}|\pi_u\wh\varphi_{ H}^{T'}(x_0,p_0,\bar{t}_0,I_0,u_0)|\\
& &+\int_0^te^{F(s+t_0+T')-F(t+t_0+T')}\Big|\langle H_p,p\rangle-H\Big|_{\wh\varphi_{ H}^{s+T'}(x_0,p_0,\bar{t}_0,I_0,u_0)}ds
\een
where the first term of the right hand side will tend to zero as $t\to +\infty$, and the second term has a uniform bound depending only on $[f], C'$. Therefore, there exists a time $T''(x_0,p_0,\bar{t}_0,I_0,u_0)$ such that for any $t\geq T'+T''$, there exists a constant $C''=C''(C',[f])$ such that 
\[
\big\|\pi_u\wh\varphi_{ H}^{t}(x_0,p_0,\bar{t}_0,I_0,u_0)\big\|\leq C''.
\]
Benefiting from the boundedness of $u-$component, we can repeat aforementioned scheme to the $I-$equation of (\ref{eq:dis}), then prove (\ref{eq:*}).\qed
%there must be a time $T(x_0,p_0,t_0,I_0,u_0)$ and a constant $C(u_0,\pi_{T^*M\Om})$
\end{proof}

Once $\Om$ is compact, it has to be the maximal global attractor of $\wh\varphi_{ H}^t$ in the whole phase space $T^*M\times T^*\T\times\R$. 
%\begin{lem}
%$\Om$ is the maximal invariant set contained in $\Sigma_{\wh H}$.
%\end{lem}
Then due to Proposition \ref{Sec3:pro_differentiable}, any backward calibrated curve $\gamma_{x,s}^-:(-\infty,s]\rightarrow M$ decides a unique trajectory 
\ben
\wh \varphi_{ H}^t&\Big(&\cL^{-1}(x,\lim_{\varsigma\rightarrow s_-}\dot\gamma_{x,s}^-(\varsigma),s),\alpha-f(s)u_\alpha^-(x,s)\\
& &-H\big(\cL^{-1}(x,\lim_{\varsigma\rightarrow s_-}\dot\gamma_{x,s}^-(\varsigma),s)\big),u_\alpha^-(x,s)\Big)
\een
for $t\in\R$, which lies in $\wh\Sigma_{ H}$. Furthermore, 
\[
\wh\cA:=\Big\{\Big(\cL^{-1}(x,\partial_x u_\alpha^-(x,t),t),\partial_t u_\alpha^-(x,t),u_\alpha^-(x,t)\Big)\Big|(x,t)\in\cA\Big\}\subset\Om
\]
 because $\Om$ is the maximal invariant set in $\wh\Sigma_{ H}$.

\begin{lem}
$\wh\cA$ is the maximal invariant set contained in $\wh\Sigma_{H}^0$.
\end{lem}
\begin{proof}
If $\cI$ is an invariant set contained in $\wh\Sigma_{ H}^0$, then $\pi_u(\wh\varphi_{ H}^t(\cI))$ is always bounded. Due to Lemma \ref{lem:layer}, any trajectory in $\cI$ has to be backward calibrated. As $\cI$ is invariant, any trajectory in it has to be contained in $\wh\cA$. \qed
\end{proof}

\vspace{20pt}

\noindent{\it Proof of Theorem \ref{cor:critical}:} Let $\tilde{\mu}\in \mathfrak{M}_L$ be ergodic, then we can find $(x_0,v_0,t_0)\in TM\times\mathbb{T}$ such that
\ben
& &\int_{TM\times\mathbb{T}}e^{F(t)}(L(x,v,t)+c(H))\mbox{d}\tilde{\mu}\\
&=&\lim_{T\to+\infty}\frac{1}{T}
\int^0_{-T}e^{F(\tau)}(L(\varphi_L^\tau(x_0,v_0,t_0))+c(H))\mbox{d}\tau.
\een
Therefore, for any weak KAM solution $u_c^-:M\times\T\rightarrow\R$ of \eqref{eq:sta-hj2}, we have
\ben
& &e^{F(0)}u_c^-(x_0,t_0)-e^{F(-T)}u_c^-(\pi_{x,t}\varphi_L^{-T}(x_0,v_0,t_0))\\
&\leq&\int^0_{-T}e^{F(\tau)}(L(\varphi_L^\tau(x_0,v_0,t_0))+c(H))\mbox{d}\tau,
\een
which implies
\ben
	& &\lim_{T\to+\infty}\frac{1}{T}\int^0_{-T}e^{F(\tau)}(L(\varphi_L^\tau(x_0,v_0,t_0))+c(H))\mbox{d}\tau\\
	&=&\lim_{T\to+\infty}\frac{1}{T}
(e^{F(0)}u_c^-(x_0,t_0)-e^{F(-T)}u_c^-(\pi_{x,t}\varphi_L^{-T}(x_0,v_0,t_0))=0.
\een
Hence,
$$
\int_{TM\times\mathbb{T}}e^{F(t)}(L(x,v,t)+c(H))\mbox{d}{\tilde{\mu}}\geq 0.
$$
That further implies 
%Note that $F(\tau)$ is 1-periodic. We derive 
%$$
%\int_{TM\times\mathbb{T}}e^{F(\tau)}\mbox{d}\tilde{\mu}=
%\lim_{n\to\infty}\frac{1}{n}\int^n_0e^{F(\tau)}\mbox{d}\tau=\int^1_0e^{F(\tau)}\mbox{d}\tau.
%$$
%By {\sf Ergodic Decomposition Theorem},  for every $\wt{m}\in\mathfrak{M}_L$, 
%$$
%\int_{TM\times\mathbb{T}}e^{F(t)}(L(x,v,t)+c)\mbox{d}\wt{m}\geq 0\mbox{ and }
%\int_{TM\times\mathbb{T}}e^{F(\tau)}\mbox{d}\wt{m}=\int^1_0e^{F(\tau)}\mbox{d}\tau
%$$
%which implies
$$
\frac{\inf_{\mu\in\mathfrak M_L}\int_{TM\times\mathbb{T}}e^{F(t)}L(x,v,t)\mbox{d}\tilde{\mu}}{
\int_{0}^1e^{F(\tau)}\mbox{d}\tau}\geq-c(H).
$$
On the other side, for any $(x,0)\in M\times\mathbb{T}$ fixed, the backward calibrated curve $\gamma^-_{x,0}:(-\infty,0]\to M$ satisfies
\begin{align*}
e^{F(0)}u_c(\gamma^-_{x,0}(0),0)&-e^{F(-n)}u_c(\gamma^-_{x,0}(-n),-n)\\
&=\int^{0}_{-n}e^{F(\tau)}(L(\gamma^-_{x,0}(\tau),\dot{\gamma}^-_{x,0}(\tau),\tau)+ c)\mbox{d}\tau
\end{align*}
for any $n\in\Z_+$.
By the {\sf Resize Representation Theorem}, the time average w.r.t. $\gamma^-_{x,0}|_{[-n,0]}:[-n,0]\rightarrow M$ decides a sequence of Borel probability measures ${\mu}_n$. Due to Lemma \ref{Sec3:calibrated dot bounded}, we can always find a subsequence  $\{{\tilde{\mu}}_{n_k}\}$ converging to a unique  Borel probability measure ${\tilde{\mu}}^*$, i.e.
\ben
\int_{TM\times\mathbb{T}}g(x,v,t)\mbox{d}{\tilde{\mu}}^*
&=&\lim_{k\to\infty}\int_{TM\times\mathbb{T}}g(x,v,t)\mbox{d}{\tilde{\mu}}_{n_k}\\
&=&\lim_{k\to\infty}\frac{1}{n_k}\int^0_{-n_k}g(\gamma^-_{x,0}(\tau),\dot{\gamma}^-_{x,0}(\tau),\bar{\tau})\mbox{d}\tau
\een
for any $g\in C_c(TM\times\T,\R)$. Besides, we can easily prove that $\tilde{\mu}^*\in\mathfrak M_L$ and 
%Note that $(\gamma_*(\tau),\dot{\gamma}_*(\tau),\tau)$ solves  (\ref{eq:e-l}). We extend the solution from $(-\infty,0]$ to $\mathbb{R}$, denoted still by $(\gamma_*(\tau),\dot{\gamma}_*(\tau),\tau)$.
%
%
%
%For each $s>0$,
%\begin{align*}
%	\int_{TM\times\mathbb{T}}g\circ\varphi_{L}^s\mbox{d}\wt{\mu}^*
%&=\lim_{k\to\infty}\frac{1}{n_k}\int^s_{-n_k+s}g(\gamma_*(\tau),\dot{\gamma}_*(\tau),\{\tau\})\mbox{d}\tau\\
%&=\lim_{k\to\infty}\frac{1}{n_k}\int^0_{-n_k}g(\gamma_*(\tau),\dot{\gamma}_*(\tau),\{\tau\})\mbox{d}\tau\\
%&=\int_{TM\times\mathbb{T}}g(x,v,t)\mbox{d}\wt{\mu}^*
%\end{align*}
%Hence, $\wt{\mu}^*\in\mathfrak{M}_L$. On the other hand,
\ben
	& &\int_{TM\times\mathbb{T}}e^{F(t)}(L(x,v,t)+c(H))\mbox{d}{\tilde{\mu}}^*\\
	&=&\lim_{k\to\infty}\frac{1}{n_k}\int^{0}_{-n_k}e^{F(\tau)}(L(\gamma^-_{x,0}(\tau),\dot{\gamma}^-_{x,0}(\tau),\tau)+ c(H))\mbox{d}\tau\\
	&=&\lim_{k\to\infty}\frac{1}{n_k}\bigg(u^-_c(\gamma^-_{x,0}(0),0)-u^-_c(\gamma^-_{x,0}(-n_k),-n_k)\bigg)=0.
\een
Then,
$$
-c(H)=\frac{\inf_{ \tilde{\mu}\in\mathfrak M_L}\int_{TM\times\mathbb{T}}e^{F(t)}L(x,v,t)\mbox{d}{\tilde{\mu}}}{\int_0^1e^{F(\tau)}\mbox{d}\tau}.
$$
Gathering all the infimum of the right side of previous equality, we get a set of  Mather measures $\mathfrak M_m$. Due to the {\sf Cross Lemma} in \cite{Mat}, the Mather set 
\[
\wt\cM:=\overline{\bigcup_{\tilde{\mu}\in\mathfrak M_m} \text{supp}(\tilde{\mu})}
\]
is a Lipschitz graph over $\cM:=\pi\wt\cM$.
%$\varphi_{L}^t-$invariant set,
%
%then for any $(x,v,s)\in\wt\cM$, there must be a globally calibrated curve $\gamma_{x,s}:\R\rightarrow M$, such that $(\gamma_{x,s}(t),\dot\gamma_{x,s}(t),t)=\varphi_L^{t-s}(x,v,s)$ for all $t\in\R$. 
%
%Due to Proposition \ref{Sec3:pro_differentiable}, $d_xu_c^-(x,s)=\partial_v L(x,v,s)$. Similar with subsection \ref{s3.1}, $\pi^{-1}:\pi\wt\cM\rightarrow\wt\cM$ is Lipschitz since $\wt\cM$ is compact in $TM\times\T$ and $\partial_vL\circ (\pi|_{\wt\cM})^{-1}=d_x u_c^-|_{\cM}$ is $C^{r-1}-$smooth.
%
%which finishes the proof.
\qed

%\begin{lem}
%The function 
%\[
%G(x,p,t,I):=I+f(t)u^-(x,t)+H(x,p,t)
%\]
% is a Lyapunov function of (\ref{eq:ode1}).
%\end{lem}
%\begin{proof}
%\end{proof}

%\begin{lem}
%$\pi_I\Om\in [-K,K]$ with $K=....$
%\end{lem}

%\begin{proof}
%\end{proof}

\section{Convergence of parameterized viscosity solutions}\label{s4}

In this section we deal with the convergence of weak KAM solution $u_\dt^-$ for system (\ref{eq:ham-par}) as $\dt\rightarrow 0_+$. 
Recall that $[f_0]=0$ and 
\[
f_1(t):=\lim_{\dt\rightarrow 0_+}\dfrac{f_\dt(t)-f_0(t)}{\dt}>0,
\]
 there must exist a $\dt_0>0$ such that 
\[
f_{\dt}(t)>f_0(t),\quad\forall\ t\in\T
\]
for all $\dt\in[0,\dt_0]$. 
%In other words, $f_\dt$ converges to $f_0$ strictly from above.
%which implies 
%\[
%\frac{f_\dt(t)-f_0(t)}{\dt}\geq 0,\quad\forall\ t\in\T,\ \dt\in(0,1].
%\]
%Moreover, 
%\[
%\int_0^1e^{F_0(t)}\lim_{\dt\rightarrow 0_+}\frac{f_\dt(t)-f_0(t)}{\dt}dt>0
%\]
%%%\[
%%%\frac{d}{d\dt}\bigg|_{\dt=0}f_\dt(t)\geq 0,\quad\forall t\in\T
%%%\]
%%and
%%%$[f_\dt] \searrow 0=[f_0]$ as $\dt\rightarrow 0_+$, 
%%\[
%%\int_0^1e^{F_0(t)}\frac{d}{d\dt}\bigg|_{\dt=0}f_\dt(t)dt>0
%%\]
%is assumed. 
Due to Theorem \ref{cor:1} there exists a unique $c(H)$, such that the weak KAM solutions $u_0^-$ of (\ref{eq:hj-criti}) with $\alpha=c(H)$ exist. 
For each $(x,t)\in M\times\mathbb{R}$ and $s<t$, 
the {\sf Lax-Oleinik operator} 
$$
T_s^{\delta,-}(x,t)=\inf_{\substack{\gamma\in C^{ac}([s,t],M)\\\gamma(t)=x}}\int^t_se^{F_\delta(\tau)-F_\delta(t)}\big(L(\gamma(\tau),\dot{\gamma}(\tau),\tau)+c(H)\big)\mbox{d}\tau
$$
is well defined, of which the following Lemma holds:
%. The minimizer curve is $C^1$ and solves Euler-Lagrangian equation (\ref{eq:e-l}).

\begin{lem}
For each $\delta\geq 0$ and $T_s^{\delta,-}(x,t)$ converges uniformly to $u^-_{\delta}(x,t)$ on each compact subset of $M\times \mathbb{R}$ as $s\to-\infty$.
\end{lem}

\proof
Let $\gamma^-_{\delta,x,t}:(-\infty,t]\to M$ be a calibrated curve  of $u_\delta^-(x,t)$. Then,
$$
e^{F_\delta(t)}u_\delta^-(x,t)=e^{F_\delta(s)}u_\delta^-(\gamma^-_{\delta,x,t}(s),s)+\int^t_se^{F_\delta(\tau)}(L(\gamma^-_{\delta,x,t},\dot{\gamma}^-_{\delta,x,t},\tau)+c(H))\mbox{d}\tau
$$
and
$$
e^{F_\delta(t)}T_s^{\delta,-}(x,t)\leq \int^t_se^{F_\delta(\tau)}(L(\gamma^-_{\delta,x,t},\dot{\gamma}^-_{\delta,x,t},\tau)+c(H))\mbox{d}\tau.
$$
Then,
\begin{equation}\label{eq:5-58}
	T_s^{\delta,-}(x,t)-u_\delta^-(x,t)\leq-e^{F_\delta(s)-F_\delta(t)}u^-_\delta(\gamma^-_{\delta,x,t}(s),s).
\end{equation}
On the other hand, let $\gamma_0:[s,t]\to M$ be a minimizer of $T_s^{\delta,-}(x,t)$. Then,
$$
e^{F_\delta(t)}T_s^{\delta,-}(x,t)=\int^t_se^{F_\delta(\tau)}(L(\gamma_0(\tau),\dot{\gamma}_0(\tau),\tau)+c(H))\mbox{d}\tau
$$
and
$$
e^{F_\delta(t)}u^-_{\delta}(x,t)-e^{F_\delta(s)}u^-_\delta(\gamma_0(s),s)\leq \int^t_se^{F_\delta(\tau)}(L(\gamma_0,\dot{\gamma}_0,\tau)+c(H))\mbox{d}\tau.
$$
Hence,
\begin{equation}\label{eq:5-59}
	u^-_\delta(x,t)-T_s^{\delta,-}(x,t)\leq e^{F_\delta(s)-F_\delta(t)}u_\delta^-(\gamma_0(s),s).
\end{equation}
From (\ref{eq:5-58}) and (\ref{eq:5-59}), it follows 
$$
|u^-_\delta(x,t)-T_s^{\delta,-}(x,t)|\leq e^{F_\delta(s)-F_\delta(t)}\max u^-_\delta,
$$
which means $T_s^{\delta,-}(x,t)$ converges uniformly to $u_\delta^-(x,t)$ on each compact subset of $M\times\mathbb{R}$.
\qed

%for $\dt\in(0,\dt_0]$, 
%
% That urges us to concern only viscosity solution $u_\dt^-$ for $\alpha_{0,\dt}\equiv c(H)$ (given in Theorem \ref{thm:1}), even though $\cD_{0,\dt}=\R$ for all $\dt>0$.

\begin{lem}\label{lem:equi-lip}
$u_\dt^-:M\times\T\rightarrow\R$ are equi-bounded and equi-Lipschitz w.r.t. $\dt\in(0,\dt_0]$.
\end{lem}

\proof
To show $u_{\delta}^-$ are equi-bounded from below, it suffices to show 
\[
\{T_s^{\delta,-}(x,t)|(x,t)\in M\times [0,1],s\leq 0,\delta\in(0,\delta_0]\}
\]
 is bounded from below. Let $\gamma_0:[s,t]\to M$ be a minimizer of $T_s^{\delta,-}(x,t)$, $u_\delta(\tau):=T_s^{\delta,-}(\gamma_0(\tau),\tau)$, and $\tilde{u}_\delta(\tau):=e^{F_\delta(\tau)}u_\delta(\tau), \tau\in[s,t]$. Then,
$$
\frac{\mbox{d}\tilde{u}_\delta(\tau)}{\mbox{d}\tau}=e^{F_\delta(\tau)}(L(\gamma_0(\tau),\dot{\gamma}_0(\tau),\tau)+c(H)).
$$
Hence,
$$
\frac{\mbox{d}u_\delta(\tau)}{\mbox{d}\tau}=L(\gamma_0(\tau),\dot{\gamma}_0(\tau),\tau)+c(H)-f_\delta(\tau)u_\delta(\tau).
$$
We could assume $T_s^{\delta,-}(x,t)<0$ for some $\delta\in(0,\dt_0]$, $(x,t)\in M\times [0,1],s\leq 0$, otherwise $0$ is a uniform lower bound of $\{T_s^{\delta,-}(x,t)|(x,t)\in M\times[0,1],s\leq 0,\dt\in(0,\dt_0]\}$.
Note that $u_\delta(\cdot)$ is continuous and $u_\delta(s)=0$. There exists $s_0\in[s,t)$ such that $u_\delta(s_0)=0$ and $u_\delta(\tau)<0,\tau\in(s_0,t]$.  From $f_\delta>f_0$, it follows that
$$
\frac{\mbox{d}u_\delta(\tau)}{\mbox{d}\tau}\geq L(\gamma_0(\tau),\dot{\gamma}_0(\tau),\tau)+c(H)-f_0(\tau)u_\delta(\tau),\tau\in[s_0,t].
$$
Hence,
$$
\frac{\mbox{d}}{\mbox{d}\tau}\big(e^{F_0(\tau)}u_\delta(\tau)\big)\geq e^{F_0(\tau)}(L(\gamma_0(\tau),\dot{\gamma}_0(\tau),\tau)+c(H)),
$$
where $F_0(\tau)=\int^\tau_0f_0(\sigma)\mbox{d}\sigma$.
Integrating on $[s_0,t]$, it holds that
\begin{equation}\label{eq:lowerbd}
	e^{F_0(t)}\cdot u_\delta(t)\geq\int^t_{s_0}e^{F_0(\tau)}(L(\gamma_0(\tau),\dot{\gamma}_0(\tau),\tau)+c(H))\mbox{d}\tau.
\end{equation}
Let $\beta:[t,t+2-\overline{t-s_0}]\to M$ be a geodesic with $\beta(t)=\gamma_0(t),\beta(t+2-\overline{t-s_0})=\gamma_0(s_0)$, and 
$$
|\dot{\beta}(\tau)|=\frac{d(\gamma_0(s_0),\gamma_0(t))}{2-\overline{t-s_0}}\leq \mbox{diam}(M)=:k_1.
$$
 Due to the definition of $c(H)$ in (\ref{eq:def_critical}), we derive
\begin{align*}
&\int^t_{s_0}e^{F_0(\tau)}(L(\gamma_0(\tau),\dot{\gamma}_0(\tau),\tau)+c(H))\mbox{d}\tau\\
&+\int^{t+2-\overline{t-s_0}}_te^{F_0(\tau)}(L(\beta(\tau),\dot{\beta}(\tau),\tau)+c(H))\mbox{d}\tau\geq 0.
\end{align*}
Note that 
\ben
& &\int^{t+2-\overline{t-s_0}}_te^{F_0(\tau)}(L(\beta(\tau),\dot{\beta}(\tau),\tau)+c(H))\mbox{d}\tau\\
&\leq&\int^{t+2-\overline{t-s_0}}_te^{F_0(\tau)}(C_{k_1}+c(H))\mbox{d}\tau
\leq 2(C_{k_1}+c(H))e^{\max_{t\in\T} F_0(t)}.
\een
Hence,
$$
\int^t_{s_0}e^{F_0(\tau)}(L(\gamma_0(\tau),\dot{\gamma}_0(\tau),\tau)+c(H))\mbox{d}\tau
\geq -2(C_{k_1}+c(H))e^{\max F_0}.
$$
Combining (\ref{eq:lowerbd}), we derive
$$
u_\delta(t)\geq -2 |C_{k_1}+c(H)| e^{\max F_0-\min F_0}.
$$

Next, we prove $u_{\delta}^-(x,t)$ are equi-bounded from above. It suffices to show 
$\{T^{\delta,-}_s(x,t)|(x,t)\in M\times [0,1],s\leq 0,\dt\in(0,\dt_0]\}$ is bounded from above.  We could assume $T^{\delta,-}_s(x,t)>0$ for some $\delta\in(0,\dt_0]$, $(x,t)\in M\times [0,1],s\leq 0$, otherwise $0$ is a uniform upper bound of $\{T^{\delta,-}_s(x,t)|(x,t)\in M\times[0,1],s\leq 0,\dt\in(0,\dt_0]\}$.

Let $u_0^-(x,t)$ be a weak KAM solution of
$$
\partial_tu+H(x,\partial_xu,t)+f_0(t)u=c(H),
$$
 and $\gamma^-_{x,t}:(-\infty,t]\to M$ be a calibrated curve of $u_0^-(x,t)$. Let
 $$
 v_\delta(\tau):=T^{\delta,-}_s(\gamma^-_{x,t}(\tau),\tau),\tau\in[s,t].
 $$
Then
\ben
& &\frac{e^{F_\delta(\tau+\Delta \tau)}v_\delta(\tau+\Delta \tau)-e^{F_\delta(\tau)}v_\delta(\tau)}{\Delta \tau}\\
&\leq&\frac{1}{\Delta \tau}\int^{\tau+\Delta \tau}_{\tau}e^{F_\delta(\sigma)}(L(\gamma^-_{x,t}(\sigma),\dot{\gamma}^-_{x,t}(\sigma),\sigma)+c(H))\mbox{d}\sigma.
\een
Note that
\begin{align*}
	&\varlimsup_{\Delta \tau\to 0}\frac{e^{F_\delta(\tau+\Delta \tau)}v_\delta(\tau+\Delta \tau)-e^{F_\delta(\tau)}v_\delta(\tau)}{\Delta \tau}\\
&=\varlimsup_{\Delta \tau\to 0}\frac{e^{F_\delta(\tau+\Delta\tau)}v_\delta(\tau+\Delta \tau)-e^{F_\delta(\tau+\Delta\tau)}v_\delta(\tau)+e^{F_\delta(\tau+\Delta\tau)}v_\delta(\tau)-e^{F_\delta(\tau)}v_\delta(\tau)}{\Delta \tau}\\
%+\frac{e^{F_\delta(\tau+\Delta\tau)}v(\tau)-e^{F_\delta(\tau)}v(\tau)}{\Delta\tau}\bigg)\\
&=e^{F_\delta (\tau)}\varlimsup_{\Delta \tau\to 0}\bigg(\frac{v_\delta(\tau+\Delta\tau)-v_\delta(\tau)}{\Delta \tau}\bigg)+e^{F_\delta(\tau)}f_\delta(\tau)v_\delta(\tau).
\end{align*}
Hence,
$$
\varlimsup_{\Delta \tau\to 0}\bigg(\frac{v_\delta(\tau+\Delta\tau)-v_\delta(\tau)}{\Delta \tau}\bigg)\leq L(\gamma^-_{x,t}(\tau),\dot{\gamma}^-_{x,t}(\tau),\tau)+c(H)-f_\delta(\tau)v_\delta(\tau).
$$
Since $v_\delta(s)=0$ and $v_\delta(\tau)$ is continuous,
 there exists $s_1\in[s,t)$ such that $v_\delta(s_1)=0$ and $v_\delta(\tau)>0,\tau\in(s_1,t]$.

For $\tau\in (s_1,t]$,
\begin{align*}
	\varlimsup_{\Delta \tau\to 0}\bigg(\frac{v_\delta(\tau+\Delta\tau)-v_\delta(\tau)}{\Delta \tau}\bigg)&\leq L(\gamma^-_{x,t}(\tau),\dot{\gamma}^-_{x,t}(\tau),\tau)+c(H)-f_\delta(\tau)v_\delta(\tau)\\
	&\leq L(\gamma^-_{x,t}(\tau),\dot{\gamma}^-_{x,t}(\tau),\tau)+c(H)-f_0(\tau)v_\delta(\tau).
\end{align*}
Then,
\ben
& &\varlimsup_{\Delta \tau\to 0}\bigg(\frac{e^{F_0(\tau+\Delta\tau)}v_\delta(\tau+\Delta\tau)-e^{F_0(\tau)}v_\delta(\tau)}{\Delta \tau}\bigg)\\
&\leq& e^{F_0(\tau)}(L(\gamma^-_{x,t}(\tau),\dot{\gamma}^-_{x,t}(\tau),\tau)+c(H)).
\een
%Let $\tilde{v}:[s_1,t]\rightarrow M$ be the solution of 
%$$
%\begin{cases}
%	\dot{\tilde{v}}=e^{F_0(\tau)}(L(\gamma^-_{x,t}(\tau),\dot{\gamma}^-_{x,t}(\tau),\tau)+c(H)),\\
%	\tilde{v}(s_1)=0.
%\end{cases}
%$$
From $v_\delta(s_1)=0$, it follows that
\begin{align*}
	e^{F_0(t)}v_\delta(t)&\leq\int^t_{s_1}e^{F_0(\tau)}(L(\gamma^-_{x,t},\dot{\gamma}^-_{x,t},\tau)+c(H))\mbox{d}\tau\\
	&=e^{F_0(t)}u_0^-(x,t)-e^{F_0(s_1)}u_0^-(\gamma^-_{x,t}(s_1),s_1).
\end{align*}
Then,
$$
v_\delta(t)\leq 2\max|u_0^-|\cdot e^{\max F_0-\min F_0}.
$$
%where $P=\max\{ |u_0^-(x,t)|\ |(x,t)\in M\times[0,1]\}$.

Note that $u_\delta^-(x,t)$ is equi-bounded. By a similar approach of the proof of Lemma \ref{lem:lip-dis}, we derive that $u_\delta^-$ is equi-Lipschitz. 
\qed

\begin{lem}\label{lem:com-cal}
For any $\dt\in(0,\delta_0]$ and any $(x,\bar{s})\in M\times\T$, the backward calibrated curve $\gamma_{\dt,x,s}^-:(-\infty,s]\rightarrow M$ associated with $u_\dt^-$ has a uniformly bounded velocity, i.e. there exists a constant $K>0$, such that 
$$
|\dot\gamma_{\dt,x,s}^-(t)|\leq K, \quad \forall \dt\in(0,1]\mbox{ and } t\in(-\infty,s).
$$ 
\end{lem}

\proof
By a similar way in the proof of Lemma \ref{Sec3:calibrated dot bounded}, there exists $s_0$ in each interval with length 1, such that
$$
|\dot{\gamma}^-_{\delta,x,s}(s_0)|\leq C_{k_1}+C(1),
$$
where $k_1=\mbox{diam}(M)$. 
Note that $f_{\delta}$ depends continuously on $\delta$ and is 1-periodic. We derive the Lagrangian flow $(\gamma^-_{\delta,x,s}(\tau),\dot{\gamma}^-_{\delta,x,s}(\tau),\tau)$ is 1-periodic and depends continuously on the parameter $\delta$. Hence, there exists $K>0$ depends only on $L, k_1$, and $\delta_0$,  such that $|\dot{\gamma}^-_{\delta,x,s}|<K$.\qed

\begin{prop}\label{prop:geq}
For any ergodic measure $\tilde{\mu}\in\mathfrak M_m(0)$ and any $0<\dt\leq \dt_0$, we have
%accumulating function $u_0^-\in\varlimsup_{\dt\rightarrow 0_+} u_\dt^-$, we have 
\be
\int_{TM\times\T}e^{F_0(t)}\frac{f_\dt(t)-f_0(t)}{\dt}u_\dt^-(x,t)\mbox{d}\tilde{\mu}(x,v,t)\leq 0.
\ee
\end{prop}
\begin{proof}
Since $\{u_\dt^-\}_{\dt\in(0,\dt_0]}$ is uniformly bounded and $[f_0]=0$, then 
\[
\lim_{T\rightarrow+\infty}\frac 1T\int_0^Tu_\dt^-(\gamma(t),t)\mbox{d} e^{F_0(t)}=\int_{TM\times\T}u_\dt^-(x,t)f_0(t)e^{F_0(t)}\mbox{d}\tilde{\mu}(x,v,t)
\]
%\[
%\lim_{T\rightarrow+\infty}\frac{u_\dt^-(\gamma(T))-u_\dt^-(\gamma(0))}{T}=0
%\]
for any regular curve $\wt\gamma(t)=(\gamma(t),\bar{t}):t\in\R\rightarrow M\times\T$ contained in $\cM(\dt)$. Due to Proposition \ref{Sec3:pro_differentiable},
\ben
& &\frac 1T\int_0^Tu_\dt^-(\gamma(t),t)\mbox{d} e^{F_0(t)}\\
&=&\frac 1Tu_\dt^-(\gamma(t),t) e^{F_0(t)}\Big|_0^T-\frac 1T\int_0^T  e^{F_0(t)}\big[\partial_t u_\dt^-(\gamma(t),t)+\langle \dot\gamma(t),\partial_xu_\dt^-(\gamma(t),t)\rangle\big]\mbox{d}t
\een
and 
\ben
& &\frac 1T\int_0^T  e^{F_0(t)}\big[\partial_t u_\dt^-(\gamma(t),t)+\langle \dot\gamma(t),\partial_xu_\dt^-(\gamma(t),t)\rangle\big]\mbox{d}t\\
&\leq&\frac 1T\int_0^T  e^{F_0(t)}\big[L(\gamma,\dot\gamma,t)+H(\gamma(t),\partial_x u_\dt^-(\gamma(t),t),t)+\partial_t u_\dt^-(\gamma(t),t)\big]\mbox{d}t\\
&\leq&\frac 1T\int_0^T  e^{F_0(t)}\big[L(\gamma,\dot\gamma,t)+c(H)-f_\dt(t)u_\dt^-(\gamma(t),t) \big]\mbox{d}t,
\een
by taking $T\rightarrow +\infty$ and dividing both sides by $\dt$ we get the conclusion.\qed
\end{proof}

\begin{defn}
Let's denote by $\cF_-$ the set of all viscosity subsolutions $\om:M\times\T\rightarrow \R$ of (\ref{eq:hj-par}) with $\dt=0$ such that 
\be
\int_{TM\times\T} f_1(t)\om(x,t)e^{F_0(t)}\mbox{d}\tilde{\mu}\leq 0,\quad\forall\ \tilde{\mu}\in\mathfrak M_m(0).
\ee
\end{defn}
%\begin{rmk}
%Recall that Proposition \ref{prop:geq} indicates that $\cF_-$ is nonempty since any $u_0^-\in\varlimsup_{\dt\rightarrow 0_+}\{u_\dt^-\}$ is an element of $\cF_-$.
%%; Notice that here we assume a multiplier $\int_{TM}\frac{\partial L}{\partial u}(x,v,0)d\mu$ to ensure the integral is about a probability measure.
%\end{rmk}
\begin{lem}
The set $\cF_-$ is uniformly bounded from above, i.e.
\[
\sup\{u(x)|\ \forall\ x\in M,\ u\in\cF_-\}<+\infty.
\]
\end{lem}
\begin{proof}
By an analogy of Lemma 10 of \cite{CIM}, all the functions in the set 
\[
\Big\{ e^{F_0(t)}\om:M\times\T\rightarrow \R\Big|\om\prec_{f_0} L+c(H)\Big\}
\]
are uniformly Lipschitz with a Lipschitz constant $\kappa>0$. 
For any $\tilde{\mu}\in\mathfrak{M}_m(0)$ and $u\in\cF_-$ 
\ben
\min_{(x,t)\in M\times\T}u(x,t)e^{F_0(t)}
&=&\frac{\int_{TM\times\T} f_1(t)\min_{(x,t)\in M\times\T} u(x,t) e^{F_0(t)}\mbox{d}\tilde{\mu}}{\int_{TM\times\T}f_1(t) \mbox{d}\tilde{\mu} }\\
&=&\frac{\int_{TM\times\T} f_1(t) \min_{(x,t)\in M\times\T} u(x,t)e^{F_0(t)} \mbox{d}\tilde{\mu}}{\int_0^1f_1(t)\mbox{d}t }\\
&\leq& \dfrac{\int_{M\times\T}f_1(t) u(x,t)e^{F_0(t)} \mbox{d}\tilde{\mu}}{\int_0^1 f_1(t)  \mbox{d}t }\leq 0.
\een
%As $\min_{(x,t)\in M\times\T}u(x,t)e^{F_0(t)}\geq\min_{(x,t)\in M\times\T}u(x,t)\cdot\min_{t\in\T} e^{F_0(t)}$, then 
%\[
%\min_{(x,t)\in M\times\T}u(x,t)\leq\frac{\min_{(x,t)\in M\times\T}u(x,t)e^{F_0(t)}}{\min_{t\in\T} e^{F_0(t)}}\leq 0.
%\]
%On the other side,  
Then,
\ben
\max_{(x,t)\in M\times\T}u(x,t) e^{F_0(t)}&\leq& \max_{(x,t)\in M\times\T}u(x,t) e^{F_0(t)}-\min_{(x,t)\in M\times\T}u(x,t) e^{F_0(t)}\\
&\leq& \kappa\ \text{diam}(M\times\T)<+\infty.
\een
As a result, 
\[
\max_{(x,t)\in M\times\T} u(x,t)\leq \frac{\max_{(x,t)\in M\times\T}u(x,t) e^{F_0(t)}}{\min_{t\in\T} e^{F_0(t)}}<+\infty
\]
so we finish the proof.\qed
% Due to the compactness of $M$, we have $\max u-\min u\leq \kappa\ \text{diam}(M\times\T)<+\infty$.
\end{proof}
As $\cF_-$ is now upper bounded, we can define a supreme subsolution by
\be\label{eq:def-1}
u_0^*:=\sup_{u\in \cF_-} u.
\ee
Later we will see that this is indeed a viscosity solution of (\ref{eq:hj-par}) for $\dt=0,\alpha=c(H)$ and is the unique accumulating function of $u_\dt^-$ as $\dt\rightarrow 0_+$.

\begin{prop}\label{prop:leq}
For any  $\dt>0$, any viscosity subsolution $\om:M\times\T\rightarrow\R$ of (\ref{eq:hj-par}) with $\dt=0,\alpha=c(H)$ and any point $(x,s)\in M\times\T$, there exists a $\varphi_{L}^t-$backward invariant 
finite measure $\tilde{\mu}_{x,s}^\dt:TM\times\T\rightarrow\R$ such that 
\be
u_\dt^-(x,s)\geq \om(x,s)-\int_{TM\times\T}\om(y,t)e^{F_0(t)}f_1(t)d\tilde{\mu}_{x,s}^\dt(y,v_y,t)
\ee 
where
\ben
& &\int_{TM\times\T}g(y,t)d\tilde{\mu}_{x,s}^\dt(y,v_y,t)\\
&:=&\int_{-\infty}^s\frac{g(\gamma_{\dt,x,s}^-(t),t)\cdot \frac{\mbox{d}}{\mbox{d}t}(e^{F_\delta(t)}-e^{F_0(t)})}{f_1(t)}\mbox{d}t,\ \forall g\in C(M\times\T,\R).
\een
%\ben
%& &\int_{M\times\T}g(y,\tau)d\mu_{x,s}^\dt(y,\tau)\\
%&=&\int_{-\infty}^0g(\gamma_{\dt,x,s}^-(t),t+s) e^{\int_s^{t+s}f_\dt(\tau)d\tau}\Big(f_\dt(t+s)-f_0(t+s)\Big)dt,\quad\forall g\in C(M\times\T,\R)
%\een
%with 
%with $F_\dt(t):=\int_0^tf_\dt(s)ds$. 
\end{prop}
\begin{proof}
For any $(x,\bar{s})\in M\times\T$ and any $\dt\in(0,\dt_0]$, there exists a backward calibrated curve $\gamma_{\dt,x,s}^-:(-\infty,s]\rightarrow M$ ending with $x$, such that the viscosity solution $u_\dt^-$ is differentiable along $(\gamma_{\dt,x,s}^-(t),\overline{t})$ for all $t\in(-\infty,s)$ due to Proposition \ref{Sec3:pro_differentiable}. Precisely, for all $t\in(-\infty,s)$
$$
\frac{\mbox{d}}{\mbox{d}t}\big(e^{F_\delta}(t)u_\delta^-(\gamma^-_{\delta,x,s}(t),t)\big)=e^{F_\delta(t)}\big(L(\gamma_{\delta,x,x}(t),\dot{\gamma}_{\delta,x,s}(t),t)+c(H)\big).
$$
Integrating on $[s,-T]$,
\ben
& &e^{F_\dt(s)}u_\dt^-(x,s)-e^{F_\dt(-T)}u_\dt^-(\gamma_{\dt,x,s}^-(-T),-T)\\
&=&\int_{-T}^se^{F_\dt(t)}\Big[L\Big(\gamma_{\dt,x,s}^-(t),\dot\gamma_{\dt,x,s}^-(t),t\Big)+c(H)\Big]dt
\een
for any $T>0$, where $F_\dt(t):=\int_0^tf_\dt(\tau)d\tau$. On the other side, 
\[
\partial_t\om (x,t)+H(x,\partial_x\om(x),t)+f_0(t)\om(x,t)\leq c(H),\quad a.e.\ (x,\bar{t})\in M\times\T
\]
since $\om$ is also a subsolution of  (\ref{eq:hj-par}) (with $\dt=0$), then
\ben
& &e^{F_\dt(s)}u_\dt^-(x,s)-e^{F_\dt(-T)}u_\dt^-(\gamma_{\dt,x,s}^-(-T),-T)\\
&\geq&\int_{-T}^se^{F_\dt(t)}\Big[L\Big(\gamma_{\dt,x,s}^-(t),\dot\gamma_{\dt,x,s}^-(t),t\Big)+H\Big(\gamma_{\dt,x,s}^-(t),\partial_x\om(\gamma_{\dt,x,s}^-(t),
t),t\Big)\\& &+\partial_t\om(\gamma_{\dt,x,s}^-(t),t)+f_0(t)\om(\gamma_{\dt,x,s}^-(t),t)\Big]\mbox{d}t\\
&\geq&\int_{-T}^se^{F_\dt(t)}\Big[
%\langle \dot\gamma_{\dt,x,s}^-(t),\partial_x\om(\gamma_{\dt,x,s}^-(t),t+s)\rangle+
\frac{d}{dt}\om(\gamma_{\dt,x,s}^-(t),t)+f_0(t)\om(\gamma_{\dt,x,s}^-(t),t)\Big]\mbox{d}t\\
&\geq&e^{F_\dt(s)} \om(x,s)-e^{F_\dt(-T)}\om(\gamma_{\dt,x,s}^-(-T),-T)\\
& &-\int_{-T}^s\om(\gamma_{\dt,x,s}^-(t),t) e^{F_\dt(t)}\Big(f_\dt(t)-f_0(t)\Big)\mbox{d}t.
\een
%\[
%e^{F_\dt(s)} \om(x,s)-e^{F_\dt(-T+s)}\om(\gamma_{\dt,x,s}^-(-T),s-T)\leq\int_{-T}^0e^{F_\dt(t+s)}\Big[L\Big(\gamma_{\dt,x,s}^-(t),\dot\gamma_{\dt,x,s}^-(t),t+s))\Big)+c(H)\Big]dt.
%\]
%Let $\delta>0$. According to Proposition (\ref{prop:sub-sol}), there exists a smooth function $w_\delta$ such that 
%$$
%\|\omega-\omega_\delta\|<\delta \mbox{ and } H(x,\partial_x\omega_\delta(x),0)<c(H)+\delta.
%$$
%Combining previous two facts we get
%\[
%u_\dt^-(x,s)-e^{F_\dt(-T)}u_\dt^-(\gamma_{\dt,x,s}^-(-T),s-T)\geq \om(x,s)-e^{F_\dt(-T)}\om(\gamma_{\dt,x,s}^-(-T),s-T),
%\]
By taking $T\rightarrow+\infty$ we finally get 
\ben
e^{F_\dt(s)}u_\dt^-(x,s)-e^{F_\dt(s)}\om(x,s)\geq-\int_{-\infty}^s\om(\gamma_{\dt,x,s}^-(t),t) e^{F_\dt(t)}\Big(f_\dt(t)-f_0(t)\Big)\mbox{d}t.
\een
By a suitable transformation, 
\ben
& &u_\dt^-(x,s)\\
&\geq&\om(x,s)-\int_{-\infty}^s\om(\gamma_{\dt,x,s}^-(t),t)e^{F_0(t)} e^{F_\delta(t)-F_0(t)}\Big(f_\dt(t)-f_0(t)\Big)\mbox{d}t\\
&=&\om(x,s)-\int_{-\infty}^s\om(\gamma_{\dt,x,s}^-(t),t) e^{F_0(t)}\mbox{d}e^{F_\delta(t)-F_0(t)}\\
&=&\om(x,s)-\int_{-\infty}^s\om(\gamma_{\dt,x,s}^-(t),t)e^{F_0(t)}f_1(t)\frac{\mbox{d}e^{F_\delta(t)-F_0(t)}}{f_1(t)}.
\een
Then for any $g\in C(M\times\T,\R)$, the measure $\tilde{\mu}_{x,s}^\dt$ defined by 
\[
\int_{TM\times\T}g(y,\tau)\mbox{d}\tilde{\mu}_{x,s}^\dt(y,v_y,\tau):=\int_{-\infty}^sg(\gamma_{\dt,x,s}^-(t),t) \frac{\mbox{d}e^{F_\delta(t)-F_0(t)}}{f_1(t)}
\]
is just the desired one.\qed
\end{proof}

\begin{lem}\label{lem:mat-mea}
Any weak limit of the normalized measure 
\be
\wh\mu_{x,s}^\dt:=\frac{\tilde{\mu}_{x,s}^\dt}{\int_{TM\times\T}\mbox{d}\tilde{\mu}_{x,s}^\dt}
\ee
as $\dt\to0_+$ is contained in $\mathfrak M_m(0)$, i.e. a Mather measure. 
\end{lem}
\begin{proof}
As is proved in Proposition \ref{prop:leq}, $\tilde{\mu}_{x,s}^\dt$ are uniformly bounded w.r.t. $\dt\in(0,\dt_0]$. Therefore, it suffices to prove that any weak limit $\tilde{\mu}_{x,s}$ of $\tilde{\mu}_{x,s}^{\dt}$ as $\dt\rightarrow 0_+$ satisfies the following two conclusions:\medskip

First, we show $\tilde{\mu}_{x,s}$ is a closed measure. It is equivalent to show that for any $\phi(\cdot)\in C^1(M\times\T,\R)$, 
\[
\lim_{\dt\rightarrow 0_+}\int_{-\infty}^s\frac{\mbox{d}}{\mbox{d}t}\phi(\gamma_{\dt,x,s}^-(t),t)\frac{\mbox{d}e^{F_\delta(t)-F_0(t)}}{f_1(t)}=0.
\]
Indeed, we have 
\ben
& &\lim_{\dt\rightarrow 0_+}\int_{-\infty}^s\frac{\mbox{d}}{\mbox{d}t}\phi(\gamma_{\dt,x,s}^-(t),t)\frac{\mbox{d}e^{F_\delta(t)-F_0(t)}}{f_1(t)}\\
&=&\lim_{\dt\rightarrow 0_+}\int_{-\infty}^s e^{F_\delta(t)-F_0(t)} \frac{f_\delta(t)-f_0(t)}{f_1(t)}\mbox{d}\phi(\gamma_{\dt,x,s}^-(t),t)\\
&=&\lim_{\dt\rightarrow 0_+}\frac{f_\dt(t)-f_0(t)}{f_1(t)}e^{F_\delta(t)-F_0(t)}\phi(\gamma_{\dt,x,s}^-(t),t)\Bigg|_{-\infty}^s\\
& &-\lim_{\dt\rightarrow 0_+}\int_{-\infty}^s\phi(\gamma_{\dt,x,s}^-(t),t)\cdot \mbox{d}\Big(\frac{f_\dt(t)-f_0(t)}{f_1(t)}e^{F_\delta(t)-F_0(t)}\Big)=0
\een
because $f_\dt\rightarrow f_0$ uniformly as $\dt\rightarrow 0_+$.\medskip

Next, we can show that 
\ben
%\label{eq:math-act}
\lim_{\dt\rightarrow 0_+}\int_{-\infty}^se^{F_\dt(t)}\Big[L\Big(\gamma_{\dt,x,s}^-(t),\dot\gamma_{\dt,x,s}^-(t),t\Big)+c(H)\Big]\frac{\mbox{d}e^{F_\delta(t)-F_0(t)}}{f_1(t)}=0.
\een
Note that 
$$
\frac{\mbox{d}}{\mbox{d}t}\big(e^{F_\delta(t)}u^-_\delta(\gamma^-_{\delta,x,s}(t),t)\big)=
e^{F_\delta(t)}\big(L(\gamma^-_{\delta,x,s}(t),\dot{\gamma}_{\delta,x,x}^-(t),t)+c(H)\big).
$$
We derive
\ben
& &\lim_{\dt\rightarrow 0_+}\int_{-\infty}^se^{F_\dt(t)}\Big[L\Big(\gamma_{\dt,x,s}^-(t),\dot\gamma_{\dt,x,s}^-(t),t\Big)+c(H)\Big]\frac{\mbox{d}e^{F_\delta(t)-F_0(t)}}{f_1(t)}\\
&=&\lim_{\dt\rightarrow 0_+}\int_{-\infty}^s\frac {\mbox{d}}{\mbox{d}t}\Big( e^{F_\dt(t)}u_\dt^-(\gamma_{\dt,x,s}^-(t),t)\Big)\frac{\mbox{d}e^{F_\delta(t)-F_0(t)}}{f_1(t)}=0,
\een
since $u_\dt^-$ is differentiable along $(\gamma_{\dt,x,s}^-(t),\bar t)$ for all $t\in(-\infty,s)$ and $\tilde{\mu}_{x,s}$ is closed. 
So we finish the proof.\qed
\end{proof} 

\medskip

\noindent{\it Proof of Theorem \ref{thm:3}:} Due to the stability of viscosity solution (see Theorem 1.4 in \cite{CEL}), any accumulating function $u_0^-$ of $u_\dt^-$ as $\dt\rightarrow 0_+$ is a viscosity solution of (\ref{eq:hj-par}) with $\dt=0$. Therefore, Proposition \ref{prop:geq} indicates $u_0^-\in\cF_-$, so $u_0^-\leq u_0^*$. On the other side,  Proposition \ref{prop:leq} implies $u_0^-\geq \om$ for any $\om\in\cF_-$ as $\dt\rightarrow 0_+$, since any weak limit of $\wh \mu_{x,s}^\dt$ as $\dt\rightarrow 0_+$ proves to be a Mather measure in Lemma \ref{lem:mat-mea}. So we have $u_0^-\geq u_0^*$.\qed

\section{Asymptotic behaviors of trajectories of 1-D mechanical systems}\label{s5}

 \begin{lem}\label{lem:conti}
For system (\ref{eq:ode0}), $\rho(c)$ is continuous of $c\in H^1(\T,\R)$.
 \end{lem}
\proof
Firstly, all the orbits in $\wt\cA(c)$ should have the unified rotation number. This is because $\pi^{-1}:\cA(c)\rightarrow \wt\cA(c)$ is a Lipschitz graph and dim$(M)=1$. Secondly, $\varlimsup_{c'\rightarrow c}\wt\cA(c')\subset \wt\cA(c)$ due to Lemma \ref{lem:semi-con}. That further indicates $\lim_{c'\rightarrow c}\rho(c')=\rho(c)$. \qed
%\begin{lem}[Theorem 2.4.4 of \cite{Mo}]
%Suppose $\gamma(t)$ is a global minimizer, then for all $t\in\R$, $m\in\Z$, we have 
%\[
%|x(t+m)-x(t)-m\om|\leq 1
%\]
%with $\om$ is the uniquely defined rotation number.
%\end{lem}
\begin{lem}\label{lem:om-high}
For system (\ref{eq:ode0}),  the rotation number $\rho(c)$ can be dominated by 
\be
-\|V\|_{C^1}\cdot\varsigma-c\leq \rho(c)\leq \|V\|_{C^1}\cdot\varsigma-c
\ee
%and
%\be
%\quad\quad\sqrt{c^2-2\|V\|_{C^1}(2+\lb^{-1})}\leq |\rho(c)|\leq \sqrt{c^2+2\|V\|_{C^1}(2+\lb^{-1})} \quad\;\text{(outer)}
%\ee
where  $\varsigma=\varsigma([f])>0$ tends to infinity as $[f]\rightarrow 0_+$.
\end{lem}

\proof
Recall that
\[
\dot p=-V_x(x,t)-f(t)p, 
\]
then starting from any point $(x_0,p_0,\bar{t}_0)\in T^*M\times\T$, we get 
\[
p(t)=e^{-F(t)}p_0-e^{-F(t)}\int_0^te^{F(s)}V_x(x(s),s)\mbox{d}s, \quad t>0.
\]
As $t\rightarrow +\infty$, we have
\be
\lim_{t\rightarrow+\infty}|p(t)|&\leq& \|V\|_{C^1}\cdot \limsup_{t\rightarrow+\infty}e^{-F(t)}\int_0^te^{F(s)}\mbox{d}s\nonumber\\
&\leq & \varsigma(\|f\|) \cdot \|V\|_{C^1}
\ee
for a constant $\varsigma(\|f\|)>0$ depending only on $f$. As a consequence, 
\be\label{eq:ineq-momen}
-\|V\|_{C^1}\cdot \varsigma \leq\pi_p\wt \cA(c)\leq\|V\|_{C^1}\cdot\varsigma
\ee
dominates the $p-$component of $\wt\cA(c)$. \qed\bigskip

\noindent{\it Proof of Theorem \ref{thm:4}:}
The first two items have been proved in previous Lemma \ref{lem:conti} and Lemma \ref{lem:om-high}. 
%As for the third bullet, it could be easily proved by Lemma \ref{lem:gene} and Lemma \ref{lem:semi-con}, and there must exists an interval $[c_-,c_+]$ containing $c_*$ such that $\wt\cA(c)$ contains a unique periodic trajectory of the rotation number $p/q$ for all $c\in[c_-,c_+]$. 
As for the third item, Lemma \ref{lem:om-high} has shown the boundedness of $p-$component of $\Om$, then due to Theorem \ref{thm:2}, we get the compactness of $\Om$.\qed

\appendix

\section{Mather measure of convex Lagrangians with $[f]=0$}\label{a1}

For a Tonelli Hamiltonian $H(x,p,t)$, the conjugated Lagrangian $L(x,v,t)$ can be established by (\ref{eq:led}), which is also Tonelli. On the other side, for $[f]=0$, the following Lagrangian 
\[
\wt L(x,v,t):=e^{F(t)}L(x,v,t),\quad (x,v,t)\in TM\times\T
\]
with 
\[
  F(t):=\int_0^tf(s)ds
\]
is still time-periodic as the case considered in \cite{Mat}. Besides, the Euler-Lagrange equation associated with $\wt L$ is the same with (\ref{eq:e-l}). So Ma\~n\'e's approach to get a Mather measure in \cite{Mn} is still available  for us. As his approach doesn't rely on the E-L flow, that supplies us with great convenience.

Let $X$ be a metric separable space. A probability measure on $X$ is a nonnegative,
countably additive set function $\mu$ defined on the $\sigma-$algebra $\mathscr B(X)$ of Borel
subsets of $X$ such that $\mu(X) = 1$. In this paper, $X=TM\times\T$. 
%Any measure on $TM\times\T$ can be projected to $M\times\T$ by
%\be\label{eq:proj-mea}
%\int_{M\times\T} g(x,t)d\pi_\sharp\mu(x,t)&:=&\int_{TM\times\T} g\circ \pi(x,v,t) d\mu(x,v,t)\ee
%for any $g\in C(M\times\T,\R)$. 
We say that a sequence
 of probability measures $\{\tilde{\mu}_n \}_{n\in\mathbb N}$ {\sf (weakly) converges to} a probability 
measure $\tilde{\mu}$ on $TM\times\T$ if
\[
\lim_{n\rightarrow+\infty}\int_{TM\times\T}h(x,v,t)d\tilde{\mu}_n(x,v,t)=\int_{TM\times\T} h(x,v,t)d\tilde{\mu}(x,v,t)
\]
for any $h\in C_c(TM\times\T,\R)$. 
%Accordingly, the projected measure $\pi_\sharp\mu_n$ weakly converges to $\pi_\sharp\mu$, since
%\ben
%\lim_{n\rightarrow+\infty}\int_{M\times\T} g(x,t)d\pi_\sharp\mu_n(x,t)&:=&\lim_{n\rightarrow+\infty}\int_{TM\times\T} g\circ \pi(x,v,t) d\mu_n(x,v,t)\nonumber\\
%&=&\int_{TM\times\T} g\circ\pi(x,v,t)d\mu(x,v,t)=:\int_{M\times\T} g(x,t)d\pi_\sharp\mu(x,t)
%\een
%for any $g\in C(M\times\T,\R)$.

\begin{defn}
A probability measure $\tilde{\mu}$ on $TM\times\T$ is called {\sf closed} if it satisfies:
\begin{itemize}
\item $\int_{TM\times\T}|v|d\tilde{\mu}(x,v,t)<+\infty$;
\item $\int_{TM\times\T}\langle \partial_x\phi(x,t),v\rangle+\partial_t\phi(x,t) d\tilde{\mu}(x,v,t)=0$ for every $\phi\in C^1(M\times\T,\R)$.
\end{itemize}
\end{defn}
Let's denote by ${\mathbb P}_c(TM\times\T)$ the set of all closed measures on $TM\times\mathbb{T}$, then the following conclusion is proved in \cite{Mn}:
\begin{thm}\label{thm:mane}
\[
\min_{\tilde{\mu}\in {\mathbb P}_c(TM\times\T)}\int_{TM}\wt L(x,v,t)d\tilde{\mu}(x,v,t)=-c(H).
\]
 Moreover, the minimizer $\tilde{\mu}_{\min}$ must be a Mather measure, i.e. $\tilde{\mu}_{\min}$ is invariant w.r.t. the Euler-Lagrange flow (\ref{eq:e-l}). 
\end{thm}
\begin{proof}
This conclusion is a direct adaption of Proposition 1.3 of \cite{Mn} to our system $\wt L(x,v,t)$, with the $c(H)$ already given in (\ref{eq:mea-var}).
\end{proof}

\section{Semiconcave functions}\label{a3}

Here we attach a series of conclusions about the semiconcave functions which can be found in \cite{Canna}, for the use of Proposition \ref{Sec3:pro_differentiable}.

\begin{defn}\label{Def:semiconcave}
	Assume $S$ is a subset of $\mathbb{R}^n$. A function $u:S\to\mathbb{R}$ is called semiconcave, if there exists a nondecreasing upper semicontinuous function $\omega:\mathbb{R}^+\to\mathbb{R}^+$ such that $\lim_{\rho\to 0^+}\omega(\rho)=0$ and
	\begin{equation}
		\lambda u(x)-(1-\lambda)u(y)-u(\lambda x+(1-\lambda)y)\leq \lambda(1-\lambda)|x-y|\omega(|x-y|),\forall \lambda\in[0,1].
	\end{equation}
\end{defn}
We call $\omega$ a modulus of semiconcavity for $u$ in $S$.
%Let $A\subset\mathbb{R}^n$ be open.

\begin{defn}[Definition 3.1.1 in \cite{Canna}]
	For any $x\in S$, the set
	$$
	D^+u(x)=\bigg\{p\in\mathbb{R}^n| \limsup_{y\to x}\frac{u(y)-u(x)-\langle p,y-x\rangle}{|y-x|}\leq 0\bigg\}
	$$
\end{defn}
is called the {\sf Fr\'echet superdifferential} of $u$ at $x$.

We shall give some properties of $D^+u(x)$, which can be found in Chapter 3 of \cite{Canna}.
\begin{prop}\label{Sec3:prop_semiconcave1}
Assume $A\subset\mathbb{R}^n$ is open.
	Let $u:A\to\mathbb{R}$ be a semiconcave function with modulus $\omega$ and $x\in A$. Then, 
	\begin{itemize}
		\item $D^+u(x)\not=\emptyset$.
		\item $D^+u(x)$ is a closed, convex set of $T_x^*A\cong\mathbb{R}^n$
		\item If $D^+u(x)$ is a singleton, then $u$ is differentiable at $x$.
		\item If $A$ is also convex, $p\in D^+u(x)$ if and only if 
	$$
	u(y)-u(x)-\langle p, y-x\rangle\leq |y-x|\omega(|y-x|)
	$$
	for each $y\in A$.
	\end{itemize}
\end{prop}
\begin{thm}[Theorem 3.2 in \cite{Canna2}]\label{Sec3:thm_semiconcave2}
	Let $u\in Lip_{loc}(\Omega\times(0,T))$ be a viscosity solution of 
	\be\label{eq:cauchy-scl}
	\partial_tu+G(x,\partial_xu,t,u)=0
	\ee
	where $G\in Lip_{loc}(\Omega\times\mathbb{R}^n\times(0,T)\times\mathbb{R},\R)$ is strictly convex in the second group of variables. Then $u$ is locally semiconcave in $\Omega\times(0,T)$.
\end{thm}
\begin{thm}[Proposition 3.3.4, Theorem 3.3.6 in \cite{Canna}]\label{thm:reachable deri}
For any $(x,t)\in \Omega\times (0,T)$, we define the {\sf reachable derivative set} of any viscosity solution $u$ of (\ref{eq:cauchy-scl}) by 
\ben
D^*u(x,t):=\big\{(p_x,p_t)=\lim_{n\rightarrow+\infty}(\partial_x u(x_n,t_n),\partial_t u(x_n,t_n))\in T^*_x\Om\times T_t^*(0,T)\big|\\
\exists\  (x_n,t_n)_{n\in\Z_+}\in\Om\times (0,T) \text{converging to $(x,t)$, at which $u$ is differentiable}\big\}. 
\een
Consequently, $D^+u(x,t)=co (D^*u(x,t))$ i.e. any superdifferential of $u$ at $(x,t)$ is a convex combination of elements in $D^*u(x,t)$.
\end{thm}

\end{document}